\documentclass[11pt]{amsart}
\usepackage[left=1in,right=1in,top=1in,bottom=1in]{geometry}
\usepackage{hyperref}
\usepackage{graphicx}
\usepackage{amssymb}
\usepackage{epstopdf}
\usepackage{amsmath}
\usepackage[dvipsnames]{xcolor}
\usepackage{bm}
\usepackage{xcolor}
\usepackage{todonotes}
\usepackage{verbatim}
\usepackage{mathtools}
\usepackage{dsfont}
\newtheorem{theorem}{Theorem}[section]
\newtheorem{lemma}[theorem]{Lemma}
\newtheorem{corollary}[theorem]{Corollary}

\newtheorem{definition}[theorem]{Definition}

\newtheorem{example}[theorem]{Example}

\theoremstyle{remark}

\newcommand{\prob}{\mathbb{P}}
\newcommand{\expect}{\mathbb{E}}
\newcommand{\N}{{\mathbb N}}
\newcommand{\Z}{{\mathbb Z}}

\renewcommand{\P}{\mathbb{P}}
\newcommand{\E}{\mathbb{E}}

\newcommand{\cJ}{\mathcal{J}}

\newcommand{\R}{\mathbb{R}}

\DeclareMathOperator{\Leb}{Leb}

\newcommand{\SG}{\mathcal{S}}

\newcommand{\zz}{\bm{z}}

\newcommand{\pu}{pos_U}
\newcommand{\ctu}{ct_U}
\newcommand{\pd}{pos_D}
\newcommand{\ctd}{ct_D}
\newcommand{\pr}{pos_R}
\newcommand{\ctr}{ct_R}
\newcommand{\pl}{pos_L}
\newcommand{\ctl}{ct_L}

\newcommand{\sq}{Sq}
\newcommand{\asq}{ASq}
\newcommand{\asqnk}{\asq(n,k)}
\newcommand{\ext}{\text{ext}}

\newcommand{\insrt}{\textbf{insert}}
\newcommand{\delete}{\textbf{delete}}
\newcommand{\rev}{\textbf{reverse}}
\newcommand{\irr}{\mathcal{R}_{irr}}
\newcommand{\regperm}{\mathcal{R}}
\newcommand{\avn}{A\!v_n}
\newcommand{\avnk}{ASq(\avn(321),k)}

\newcounter{indice}

\newcommand{\permutation}[1]{
	\setcounter{indice}{0};
	\foreach \i in {#1}
	\addtocounter{indice}{1};
	
	\addtocounter{indice}{1}
	\draw [help lines] (1,1) grid (\theindice,\theindice);
	
	\setcounter{indice}{1};
	
	\foreach \i in { #1 } {
		\draw (\theindice+.5,\i+.5) [fill] circle (.2);
		\addtocounter{indice}{1};
	}
	\addtocounter{indice}{-1};
	
}

\usepackage{todonotes}

\title[Almost square permutations are typically square]{Almost square permutations are typically square}

\author[J. Borga]{Jacopo Borga}
\author[E. Duchi]{Enrica Duchi}
\author[E. Slivken]{Erik Slivken}

\address[J. Borga]{Institut fur Mathematik, Universitat Zurich, Winterthurerstr. 190, CH-8057 Zurich, Switzerland}\email{jacopo.borga@math.uzh.ch}

\address[E. Duchi]{Universit\'e de Paris Diderot, IRIF, Batiment Sophie Germain, 75013, Paris, France}
\email{duchi@irif.fr}

\address[E. Slivken]{Universit\'{e} de Paris Diderot, LPSM, Batiment Sophie Germain, 75013, Paris, France}\email{eslivken@lpsm.paris}



\begin{document}

\maketitle

\begin{abstract}
A record in a permutation is a maximum or a minimum, from the left or from the right. The entries of a permutation can be partitioned into two types: the ones that are records are called external points, the others are called internal points. Permutations without internal points have been studied under the name of square permutations. Here, we explore permutations with a fixed number of internals points, called almost square permutations. Unlike with square permutations, a precise enumeration for the total number of almost square permutations of size $n+k$ with exactly $k$ internal points is not known.  However, using a probabilistic approach, we are able to determine the asymptotic enumeration.  This allows us to describe the permuton limit of almost square permutations with $k$ internal points, both when $k$ is fixed and when $k$ tends to infinity along a negligible sequence with respect to the size of the permutation. Finally, we show that our techniques are quite general by studying the set of $321$-avoiding permutations of size $n$ with exactly $k$ additional internal points ($k$ fixed). In this case we obtain an interesting asymptotic enumeration in terms of the Brownian excursion area. As a consequence, we show that the points of a uniform permutation in this set concentrate on the diagonal and the fluctuations of these points converge in distribution to a biased Brownian excursion.
\end{abstract}


\section{Introduction}

We look at permutations as diagrams, that is, if $n$ denotes the size of a permutation~$\sigma$, we identify $\sigma$ with the set of points $\{(i,\sigma(i))\}_{i=1}^n$. The points of a permutation can be divided into two types, internal and external.  The external points are the records of the permutation, either maximum or minimum, from the left or from the right.  The internal points are the points that are not external.  Square permutations are permutations where every point is external.  Almost square permutations are permutations with some fixed number of internal points.  We use the notation $\sq(n)$ to denote the set of square permutations of size $n$ and $\asqnk$ to denote the set of almost square permutations of size $n+k$ with exactly $n$ external points and $k$ internal points.  

Square permutations were first studied in \cite{mansour_square}, and later in \cite{duchi_square1} and \cite{ALBERT2011715}, where several approaches were used to find their generating function and to derive from it an explicit expression for $|\sq(n)|$, specifically
$$|\sq(n)|=2(n+2)4^{n-3}-4(2n-5)\binom{2n-6}{n-3}.$$
More recently in \cite{duchi_square2} the second author of the present paper devised an enumerative approach through generating trees which highlights a fast sampling procedure for uniform random elements in $\sq(n)$. Finally, a probabilistic exploration in \cite{borga2019square} by the first and the third author of the present paper, found many interesting limiting objects for uniform random permutations in $\sq(n)$.  

In \cite{ALBERT2011715} square permutations were referred to as \emph{convex permutations} and were described by pattern-avoidance. In particular, square permutations are permutations that avoid the sixteen permutations of size $5$ that have an internal point.  

We recall the definition of pattern avoidance for permutations. Let $\SG_n$ denote the set of permutations of size $n$.  For $\pi\in \SG_n$ and $\omega \in \SG_k$ we say that $\pi$ contains an occurrence of $\omega$ if there exists a subsequence $i_1 < \ldots < i_k$ such that $(\pi(i_1), \ldots, \pi(i_k))$ has the same relative order as $\omega.$  We say that $\pi$ avoids the pattern $\omega$ if it contains no occurrences of $\omega$.  We let $\avn(\omega)$ denote the set of permutations of size $n$ that avoid $\omega$ and for a collection of patterns $\mathcal{B}$ we let $\avn(\mathcal{B})$ denote the set of permutations of size $n$ that avoid every pattern in $\mathcal{B}$.  See \cite{bona,kit,vatter2014permutation} for a proper introduction to the wide range of topics related to patterns in permutations.

Almost square permutations were studied for the first time in \cite{disanto2011permutations}.  It was shown that, for fixed $k>0$, the generating function with respect to the size for $\asqnk$ is algebraic of degree $2$, and this generating function was explicitly computed for $k=1,2,3$ (see Theorem~\ref{squareInt123}). However, computations become intractable for $k>3$.  

Permutations that almost avoid a pattern were considered in \cite{brignall2009almost, griffiths2011almost} from an enumerative point of view.  A permutation is said to $k$-almost avoid a pattern (or set of patterns) if $k$ or fewer points can be deleted so that the resulting permutation avoids the pattern (or set of patterns). The notion of almost avoiding permutations has a deep relation with sorting algorithms (we refer to the introduction of \cite{brignall2009almost} for more details) which are widely studied in computer science and combinatorics.  The notion of almost avoidance used in \cite{brignall2009almost, griffiths2011almost} differs slightly from our definition of almost square permutations.  A permutation is $k$-almost square if removing exactly $k$ \emph{internal} points one obtains a square permutation. On the contrary, a permutation is a $k$-almost avoiding permutation if, removing $k$ or fewer points, \emph{either internal or external}, one obtains a permutation that belongs to the appropriate class.

We finally point out that problems similar to the ones mentioned above, this is, involving the removal of some specific atoms from a discrete structure, have been extensively considered in graph theory. Classes of graphs, defined as follows, have been studied: for a graph class $\mathcal{G}$ and an integer $k$, define $\mathcal A_k(\mathcal G)$ as the class of all graphs in which the removal of $k$ well-chosen vertices leads to $\mathcal G$. We refer to the introduction of \cite{leivaditis2019minor} for a nice overview of the literature on graph classes of the form $\mathcal A_k(\mathcal G)$ and related problems. We mention that probably the most famous instance of this kind of problems is the study of $k$-apex graphs, i.e.\ graphs that can be made planar by the removal of exactly $k$ vertices.

\bigskip

The first main result of this paper uses the approach in \cite{borga2019square,duchi_square2} to give the asymptotic enumeration of $\asqnk$.  

We write $a_n \sim b_n$ if $\lim_{n\to \infty} a_n/b_n = 1$, and $a_n = o(b_n)$ if $\lim_{n\to \infty} a_n/b_n = 0.$    

\begin{theorem}\label{approx_size}
	For $k=o(\sqrt n)$, as $n\to \infty,$
	
	\begin{equation}\label{approx_size_eq}
	|\asqnk| \sim \frac{k!2^{k+1}n^{2k+1}4^{n-3}}{(2k+1)!}\sim \frac{k!2^{k}n^{2k}}{(2k+1)!}|\sq(n)|.	
	\end{equation}
	
\end{theorem}

When $k$ grows at least as fast as $\sqrt n$ the above result fails. Nevertheless, when $k=o(n)$, we can still obtain the following weaker asymptotic expansion that determines the behavior of the exponential growth.

\begin{theorem}\label{approx_size_2}
	For $k=o(n)$, as $n\to \infty,$
	
	\begin{equation*}
	\log\left(|\asq(n,k)|\right)=\log\left(\frac{k!}{(2k+1)!}2^{k+1}n^{2k+1}4^{n-3}\right)+o(k).
	\end{equation*}
	
\end{theorem}

In order to determine the above asymptotic enumerations, we use an understanding of the geometric structure of a typical square permutation. Specifically, we use some previous results (established in \cite{borga2019square}) about the precise description of the typical shape of a large square permutation and then we find bounds on the different possible ways of adding internal points. These two results lead to the desired asymptotic enumeration, and also give the description of the typical shape of a large almost square permutation.

For the latter, we utilize the language of \emph{permutons} \cite{MR2995721}. A permuton is a probability measure on the unit square with uniform marginals. Every permutation $\sigma$ can be associated with the permuton $\mu_{\sigma}$ representing a scaled version of its diagram (see Section \ref{sect:permutons} for a precise definition). Permuton limits have been widely studied in recent years, see for instance \cite{bassino2017universal,borga2018localsubclose,permuton} (we refer to our previous article \cite{borga2019square} for a detailed description of the literature on permutons).

Given $z\in(0,1)$ we denote with $\mu^{z}$ the permuton corresponding to a rectangle in $[0,1]^2$ with corners at $(z,0), (0,z),(1-z,1)$ and $(1,1-z)$ (for a rigorous construction we refer to Section \ref{sect:muz}). In \cite{borga2019square} it was shown that the permuton limit of a uniform random square permutation is given by the random permuton $\mu^{\bm z}$, where $\bm{z}$ is chosen uniformly in the interval $(0,1)$.  Permutations in $\asqnk$ can be constructed starting from a permutation in $\sq(n)$ and adding internal points (shifting points appropriately).  Intuitively, this suggests that the permuton limit of $\asqnk$ is biased toward rectangles with larger area.  This is confirmed by the following result.

\begin{theorem}\label{fixedk_thm}
	Fix $k>0$.  Let $\zz^{(k)}$ denote the random variable in $(0,1)$ with density
	$$f_{\zz^{(k)}}(t) = (2k+1){2k \choose k} (t(1-t))^k,$$
	i.e., $\zz^{(k)}$ is beta distributed with parameters $(k+1,k+1)$.	
	If $\bm\sigma_n$ is uniform in $\asqnk$, then as $n\to \infty,$
	$$\mu_{\bm{\sigma}_n} \stackrel{d}{\longrightarrow} \mu^{\zz^{(k)}}.$$
\end{theorem}

The distribution of $\zz^{(k)}$, when $k$ increases, gives more weight around the value $1/2$ (see Fig.~\ref{fig:beta_distrib}). We therefore expect that, in the regime when $k\to\infty$ together with $n$ and $k=o(n)$, a uniform random permutation with $k$ internal points tends to $\mu^{1/2}$. The following theorem shows exactly this concentration result.

\begin{figure}[htbp]
	\begin{center}
		\includegraphics[scale=.5]{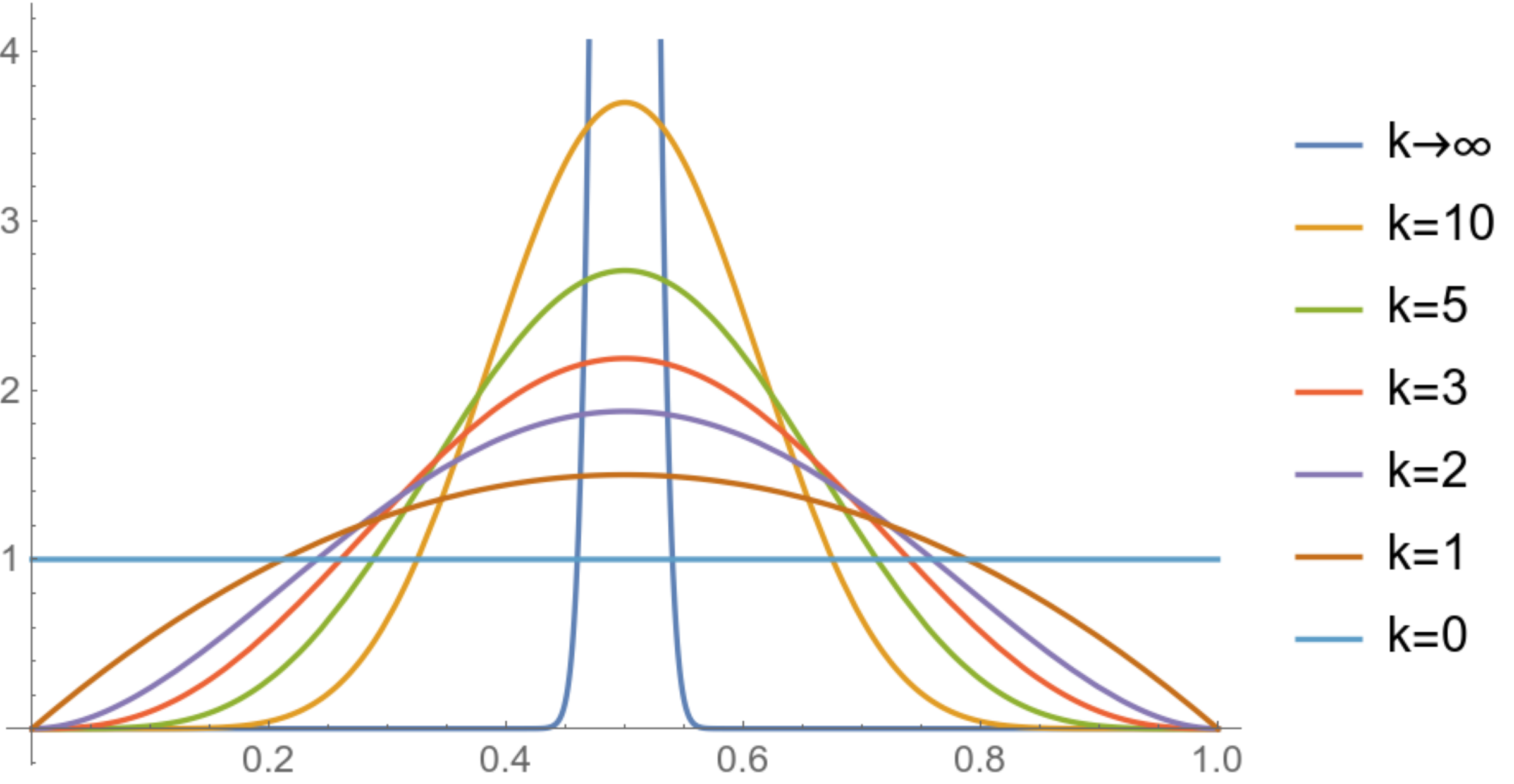}
		\caption{The chart displays the density of the distribution of $\zz^{(k)}$ for different values of $k$.}\label{fig:beta_distrib}
	\end{center}
\end{figure}

\begin{theorem}\label{permuton_limit}
	Let $k$ and $n$ both tend to infinity with $k=o(n)$. If $\bm\sigma_n$ is uniform in $\asqnk$ then
	$$\mu_{\bm{\sigma}_n} \stackrel{d}{\longrightarrow} \mu^{1/2}.$$
\end{theorem}

The probabilistic approach used to obtain our results has a wide range of possible applications and does not apply only to the set of square permutations. For example, in Section \ref{sect:321av} of this paper, we apply our techniques to establish the asymptotic enumeration for permutations avoiding the pattern $321$ with $k$ additional internal points. Permutations avoiding a decreasing sequence of length three are extensively studied in the literature, see for instance \cite{borga2018local,callan,hoffman2017pattern,HRS1,HRS2, Ja321,madras_monotone,mp, mrs}. We recall that the points of a $321$-avoiding permutation can be partitioned into two increasing subsequences, one weakly above the diagonal and one strictly below the diagonal. Therefore $321$-avoiding permutations are particular instances of square permutations.

Let $\avnk$ denote the set of permutations avoiding the pattern $321$ with $n$ external points and $k$ additional internal points or, equivalently, the subset of permutations $\sigma$ in $\asqnk$ where the pattern induced by the records of $\sigma$ is in $\avn(321)$. For fixed $k$, we show (Theorem~\ref{parallelgf123}) that the generating function of these permutations is again algebraic of degree 2, and more precisely rational in the Catalan generating series.  Explicit expressions are derived for $k=1,2,3$. As in the case of square permutations, for $k>3$, the computations to determine the generating function become intractable (see Section \ref{sect:genfunc}). Nevertheless, using our new probabilistic approach, we are able to compute the first order approximation of the enumeration.  

Let $c_n$ denote the $n$-th Catalan number, $c_n=\frac{1}{n+1}{2n\choose n}$, so that $|\avn(321)| = c_n$.    

\begin{theorem}\label{thin red line}
	Fix $k>0$.  Then as $n\to\infty,$
	
	$$|\avnk|\sim \frac{(2n)^{3k/2}}{k!}\cdot c_n\cdot  \mathbb{E}\left[\left(\int_0^1\bm e_t dt\right)^k\right],$$
	where $\bm e_t $ denotes the standard Brownian excursion on the interval $[0,1]$.
	
\end{theorem}

The evaluation on the right-hand side of the $k$-th moment of the Brownian excursion area is derived in Section 2 of \cite{janson2007brownian} where the author shows that
$$\mathbb{E}\left[\left(\int_0^1\bm e_t dt\right)^k\right]=(36\sqrt{2})^{-k}\frac{2\sqrt{\pi}}{\Gamma((3k-1)/2)}\xi_k,$$ where $\xi_k$ satisfies the recurrence
\begin{equation}\label{svante constant}
\xi_r = \frac{12r}{6r-1}\frac{\Gamma(3r+1/2)}{\Gamma(r+1/2)}- \sum_{j=1}^{r-1}{r \choose j} \frac{\Gamma(3j+1/2)}{\Gamma(j+1/2)}\xi_{r-j}, \qquad r\geq 1.
\end{equation}

The final result of our paper is a generalization of Theorem 1.2 in \cite{hoffman2017pattern} where the authors proved that the points of a uniform random 321-avoiding permutation concentrate on the diagonal and the fluctuations of these points converge in distribution to a Brownian excursion. We are able to generalize this result for uniform random permutations in $\avnk$.

We define for a permutation $\tau^k_n\in \avnk$ (with the convention that $\tau^k_n(0)=0$) and $t\in[0,1]$,
$$F_{\tau^k_n}(t) \coloneqq \frac{1}{\sqrt{2(n+k)}}\big |\tau^k_n(s(t)) - s(t) \big|,$$
where $s(t)=\max\left\{m\leq \lfloor (n+k)t \rfloor|\tau^k_n(m)\text{ is an external point}\right\}$. Note that heuristically the function $F_{\tau^k_n}(t)$ is interpolating only the external points of $\tau^k_n$, forgetting the internal ones. We also introduce the following biased Brownian excursion.

\begin{definition} \label{def:kbiasedex}
	Let $k>0$. The $k$-biased Brownian excursion $(\bm{e}^k_t)_{t\in[0,1]}$ is a random variable in the space of right-continuous functions $D([0,1],\mathbb{R})$ with the following law: for every continuous bounded functional $G:D([0,1],\mathbb{R})\to\R,$
	$$\E\left[G\left(\bm{e}^k_t\right)\right]=\mathbb{E}\left[\left(\int_0^1\bm e_t dt\right)^k\right]^{-1}\E\left[G(\bm{e}_t)\cdot\left( \int_0^1 \bm{e}_t dt\right) ^k\right],$$
	where $\bm{e}_t$ is the standard Brownian excursion on [0,1].
\end{definition}

\begin{theorem}\label{thm:fluctuations}
	Fix $k>0$. Let $\bm{\tau}^k_n$ be a uniform random permutation in $\avnk$.  Then
	$$\left(F_{\bm{\tau}^k_n}(t)\right)_{t\in [0,1]} \stackrel{d}{\longrightarrow} \left(\bm{e}^k_t\right)_{t\in [0,1]},$$
	where $\bm{e}^k_t$ is the $k$-biased Brownian excursion on $[0,1]$, and the convergence holds in the space of right-continuous functions $D([0,1],\mathbb{R})$.
\end{theorem}

\subsection*{Further questions} We collect here some problems and open questions that we would like to investigate in future projects.

\begin{enumerate}
	
	\item We studied the permuton limit for permutations with no internal points (\cite[Theorem 4.4]{borga2019square}), with a fixed number of internal points (Theorem \ref{fixedk_thm}) and with an increasing but negligible number of internal points (Theorem \ref{permuton_limit}). What can we say about the permuton limit of $\asqnk$ when $k $ has order $n$ (or is larger than $n$)? We expect a phase transition in the permuton limit from the permuton $\mu^{1/2}$ to the permuton given by the Lebesgue measure on the unit square (which is the limit of uniform permutations). We believe that the precise analysis of this phase transition is an interesting and challenging question to address.
	\item In Section \ref{sect:321av} we investigate $321$-avoiding permutations with $k$ additional internal points for $k$ fixed. What about the case when $k\to\infty$? We believe that in this regime the fluctuations of the points for a uniform permutation are not of order $\sqrt n$ any more, drastically changing the behavior of the limiting shape of the permutation. 
	\item The $k$-th moment of the Brownian excursion area appearing in Theorem~\ref{thin red line} is known to be the continuous limit of the normalized $k$-th moment of the area under large Dyck paths: it would be interesting to establish a bijection between $\avnk$ and some specific set of Dyck paths covering $k$ marked points.  
\end{enumerate}

\subsection*{Outline of the paper}

Section \ref{sect:genfunc} briefly explores the explicit generating functions for $\asqnk$ and $\asq(\avn(321),k)$.  Even for small $k$, the generating functions become quite complicated.  Section \ref{sect:insert} explains how to construct permutations in $\asqnk$ starting from a permutation in $Sq(n)$.  Section \ref{sect:perm_to_anchored_seq} recalls useful results from \cite{borga2019square} that are necessary in the subsequent sections.  Section \ref{sect:asym_enum} contains the proof of Theorems \ref{approx_size} and \ref{approx_size_2}, while Section \ref{sect:permutons} considers permutons and contains the proof of Theorems \ref{fixedk_thm} and \ref{permuton_limit}.  Finally Theorems \ref{thin red line} and \ref{thm:fluctuations} are proved in Section \ref{sect:321av}.  

\section{Generating Functions for small values of $k$.}\label{sect:genfunc}
In \cite{disanto2011permutations}, Disanto et al.\ extended the linear recursive construction of square permutations that was given in \cite{duchi_square1} in order to enumerate the class $\asqnk$. In particular they proved the following theorem for the generating function $ASq^{(k)}(t)=\sum_{n \ge 4}|\asq(n,k)| \cdot t^{n}$ of almost square permutations with $n$ external points and $k$ internal points:

\begin{theorem}[\cite{disanto2011permutations}]\label{squareInt123}
	For all $k\geq0$, the generating function $ASq^{(k)}(t)$ of almost square permutations with $k$ internal points is algebraic of degree 2 and there exists a
	rational function $R^{(k)}(u)$ such that
	\[
	ASq^{(k)}(t)= R^{(k)}(C(t)), 
	\]
	where $C(t)$ is the Catalan generating function: $C(t)=\frac{1-\sqrt{1-4t}}{2t}$.
	In particular\footnote{The expression of $R^{(k)}(u)$ in \cite{disanto2011permutations} differs by a factor $f(u)=((u-1)u^{-2})^{k-1}$, or equivalently by a factor $f(C(t))=t^{k-1}$, from the one in \eqref{con:Exact} due to the fact that the authors of \cite{disanto2011permutations} considered permutations of size $n$ with $k$ internal points (we consider permutations of size $n+k$ with $k$ internal points) and moreover, a factor $t$ is missing in their expression.}, for $k=1,2,3$,
	\begin{equation}\label{con:Exact}
	R^{(k)}(u)=\frac{8(u-1)^{4}u^{k-3}}{(2-u)^{4+4k}}P^{(k)}(u),
	\end{equation}
	where $P^{(1)}(u),P^{(2)}(u)$ and $P^{(3)}(u)$ are explicit polynomials given in \cite{disanto2011permutations}.
\end{theorem}
It is tempting to conjecture that (\ref{con:Exact}) holds for all
$k$ with some polynomial $P^{(k)}(u)$, but we would like to stress that finding an explicit expression for $R^{(k)}(u)$ is still an open problem. Indeed, even if the solution of
the system requires only subtle substitutions and applications of the
kernel method, it turns out that the computations are heavy because
they involve derivatives and specializations of series in 6 variables,
and the authors of \cite{disanto2011permutations} were not able to push them further than the case
$k=3$. However, assuming (\ref{con:Exact}), they conjectured that $P^{(k)}(C(\frac{1}{4}))= 2^{(4k-1)}k!$ where
$C(\frac{1}{4})=2$. Via singularity analysis \cite{FlSe09}, this
conjecture would imply Theorem~\ref{approx_size} for the asymptotic enumeration of almost
square permutations with $k$ internal points (for $k$ fixed). In particular it proves the cases $k=1,2,3$ of the
theorem. However, as mentioned
earlier, this direct approach appears to be intractable, as opposed to
the probabilistic approach discussed in the rest of the present paper.

\bigskip

We now discuss the  case of $\avnk$, that was not considered in \cite{disanto2011permutations}. Since the strategy is similar to the one used in \cite{disanto2011permutations}, we do not furnish here all the details and we simply explain how to adapt the computations in \cite{disanto2011permutations} to obtain our result. The explicit computations can be found in the Maple files in \cite{maplecomp2019}.

In the case of $\avnk$, the system of equations used in
\cite{disanto2011permutations} takes a slightly simpler form. Let $ASq(Av(321),k)(t)$ denote the generating function for the enumeration sequence $|ASq(Av_n(321),k)|$. This generating function decomposes according to the five relevant subclasses of permutations identified in \cite{disanto2011permutations} as
\[
ASq(Av(321),k)(t)=F^{(k)}_A(t)+F^{(k)}_{B_\alpha}(t)+ F^{(k)}_{B_{a,u}}(t)+ F^{(k)}_{B_{a,c}}(t)+F^{(k)}_{B_{d,c}}(t). \]
Introducing five catalytic variables $u,v,x,y,z$, the multivariate power series $F_E^{(k)}(u,v;x,y,z)=F_E^{(k)}(t;u,v;x,y,z)$, $E\in\{A,B_{\alpha},B_{a,u},B_{a,c},B_{d,c}\}$,  satisfy a system of linear equations of the form (see below for the notation)
\[
F_E^{(0)}(u,v;x,y,z)=t\cdot \Phi_E(\mathcal{G}^{(0)}(u,v);x,y,z)+t^2uv\cdot\delta_{E=A}+t^2\cdot\delta_{E=B_{\alpha}},
\]
and for $k\geq1$ 
\[
F_{E}^{(k)}(u,v;x,y,z)=t\cdot \Phi_E(\mathcal{G}^{(k)}(u,v);x,y,z)+\nabla(F_E^{(k-1)})(u,v;x,y,z).
\]
We now explain the various terms appearing in the previous equations (noting that the system is very similar to the one obtained in \cite[Proposition 2.1]{disanto2011permutations}). We denote with $\delta_{\mathcal{P}}$ the indicator function of a property $\mathcal{P}$. For $k\geq 0$,
$$\mathcal{G}^{(k)}(u,v)=\left\{G^{(k)}_A(u,v),G^{(k)}_{B_{\alpha}}(u,v),G^{(k)}_{B_{a,u}}(u,v),G^{(k)}_{B_{a,c}}(u,v),G^{(k)}_{B_{d,c}}(u,v)\right\},$$
where $G_E^{(k)}(u,v)=F_E^{(k)}(u,v;1,1,1)$ depends only on $u$ and $v$. Additionally, $\Phi_E(\mathcal{G}^{(k)}(u,v);x,y,z)$ is a linear combination with rational coefficients in $u,v,x,y,z$ of the series in the family $\mathcal{G}^{(k)}(a,b)$ using combinations of the substitutions $a\in\{uxy,u,xy,1\}$ and $b\in\{vz,v,z,1\}$. Finally, $\nabla$ is a mixed
differential and divided difference operator,
\begin{align*}
  \nabla(f)(u,v;x,y,z)=&\frac{x^2y^2z}{1-z}\left((\partial_xf)_{|x=xy,y=1,z=1}-(\partial_xf)_{|x=xyz^{-1}}\right)\\
  &+\frac{xy^2}{1-xy^{-1}}\left(\partial_yf-xz^{-1}(\partial_xf)_{|x=xyz^{-1},y=z}\right)\\&+\frac{xy^2z^{-1}}{(1-yz^{-1})^2}\left(f-f_{|x=xyz^{-1},y=z}\right),
\end{align*}
where $\partial_x$ (resp.\ $\partial_y$) denotes the partial derivative w.r.t. $x$ (resp. $y$).

The first system for the series $F_E^{(0)}$ can be solved explicitly
using the kernel method and yields as expected a refinement of the
Catalan generating series. In order to deal with $k\geq1$ we first
set $x=y=z=1$ and build a generic system where the terms $\nabla(F_E^{(k-1)})(u,v,1,1,1)$ are considered as parameters $h^{(k-1)}_E(u,v)$:
\[
G_E^{(k)}(u,v)=t\cdot \Phi_E(\mathcal{G}^{(k)}(u,v);1,1,1)+h^{(k-1)}_E(u,v).
\]
 We solve this system of  equations using the kernel method on the series $G_E^{(k)}(u,v)$. This yields an explicit expression of the form
\[
G_E^{(k)}(u,v)=\sum_{a,b,D} g^{E}_{a,b,D}(U,u,v)\cdot h^{(k-1)}_D(a,b),
\]
with $D\in\{A,B_{\alpha},B_{a,u},B_{a,c},B_{d,c}\}$, $b\in\{v,1\}$ and $a\in\{V,U,u,1\}$, with
$V=1/(1-v(1-U)/U^2)$, and $U=1+tU^2$, and where
the coefficients $g^E_{a,b,D}(U,u,v)$ are rational in $U,u,v$.

Returning to the original problem we have that
\[
F_{E}^{(k)}(u,v;x,y,z)=t\cdot \Phi_E\left(\sum_{a,b,D} g_{a,b,D}^E(U,u,v)\cdot h^{(k-1)}_D(a,b);x,y,z\right)+\nabla(F_E^{(k-1)})(u,v;x,y,z),
\]
where
\begin{align*}
h^{(k-1)}_D(a,b)&=\lim_{x,y,z\to1}\nabla(F_D^{(k-1)})(a,b;x,y,z))
\\&=\left(\partial_x(\partial_y+\partial_z-1)-\frac12(\partial_x^2+\partial_y^2)\right)(F_D^{(k-1)})(a,b;1,1,1).
\end{align*}
This equation implies a result similar to Theorem~\ref{squareInt123} above and it can be iterated from the initial terms
$F_E^{(0)}$ to get the successive series $F_E^{(k)}$ for
$k\geq 1$.
\begin{theorem}\label{parallelgf123}
	For all $k\geq0$, the generating function $ASq(Av(321),k)(t)$ of square permutations avoiding $321$ with $k$ additional internal points is algebraic of degree 2 and there exists a
	rational function $R^{(k)}(u)$ such that
	$$
	ASq(Av(321),k)(t)= R^{(k)}(C(t)),
	$$ 
	where $C(t)$ is the Catalan generating function: $C(t)=\frac{1-\sqrt{1-4t}}{2t}.$
	In particular, for $k=1,2,3$,
	\begin{equation*}
	R^{(k)}(u)= \frac{(u-1)^4}{(u^2+1-u)^{2k}(u-2)^{3k-1}}P^{(k)}(u),
	\end{equation*}
	where $P^{(1)}(u),P^{(2)}(u)$ and $P^{(3)}(u)$ are explicit polynomials given in \cite{maplecomp2019}. 
\end{theorem}

Via singularity analysis, this theorem implies the cases $k=1,2,3$ of Theorem~\ref{thin red line}.
But again, for $k > 3$, the computations become intractable.

\section{Internal insertions and deletions for permutations}\label{sect:insert}

To understand permutations in $\asqnk$ we investigate how they can be obtained by adding points to a permutation in $\sq(n)$.  Each permutation $\pi\in \asqnk$ has a unique \emph{exterior}, $\sigma=\ext(\pi)$, obtained by removing the internal points in $\pi$ and appropriately shifting the remaining points (keeping the relative position among them) so that the resulting set of points corresponds to a (square) permutation.  

\begin{example}
	Consider the permutation $\pi=4752316=\begin{array}{lcr}
	\begin{tikzpicture}
	\begin{scope}[scale=.3]
	\permutation{4,7,5,2,3,1,6}
	\draw (3+.5,5+.5) [red, fill] circle (.3);
	\draw (5+.5,3+.5) [red, fill] circle (.3);
	\end{scope}
	\end{tikzpicture}
	\end{array}\in ASq(5,2)$, where we highlighted in red the two internal points. Then the exterior of $\pi$ is the permutation $$\sigma=\ext(\pi)=35214=\begin{array}{lcr}
	\begin{tikzpicture}
	\begin{scope}[scale=.3]
	\permutation{3,5,2,1,4}
	\end{scope}
	\end{tikzpicture}
	\end{array}.$$
\end{example} 

In this section we define the \emph{insertion} and \emph{deletion} operation on permutations.  These operations will allow us to grow certain classes of permutations from other well-understood classes.  

For $n\in\N$, we denote with $[n]$ the set $\{1,2,\dots,n\}$. For a permutation $\sigma$ of size $n$ and a pair $(i,j) \in [n+1]^2$, the \emph{insertion} of $(i,j)$ in $\sigma$ gives the permutation obtained by adding a point at $(i,j)$ and shifting the points in $\sigma$ at or to the right of column $i$ to the right by 1 and shifting the points in $\sigma$ at or above row $j$ up by $1$.  We denote the permutation by $\insrt(\sigma,(i,j))$.  For a permutation $\sigma'$ that contains the point $(i',j')$, the \emph{deletion} of $(i',j')$ in $\sigma'$ gives the permutation obtained by removing the point $(i',j') = (i',\sigma(i'))$ and shifting points to the right of column $i'$ to the left by $1$ and points above row $j'$ down by $1$.  We denote this permutation $\delete(\sigma',(i',j'))$.  If $\sigma'$ is obtained by inserting $(i,j)$ in $\sigma$, then $\sigma$ is obtained by deleting $(i,j)$ from $\sigma'$.  Note that any pair $(i,j)\in [n+1]^2$ is a valid insertion for $\sigma\in \SG_n$, whereas only the points of the form $(i',\sigma'(i'))$ make for valid deletions in $\sigma'$.

\begin{example}
	We consider the permutation $\pi=215684793$ and we insert on it the point $(6,6)$, obtaining (as shown in Fig.~\ref{fig:insert_6_6}) the permutation $\sigma=\insrt(\sigma,(6,6))=2\,1\,5\,7\,9\,6\,4\,8\,10\,3.$
	\begin{figure}[htbp]
		\begin{center}
			\includegraphics[scale=1]{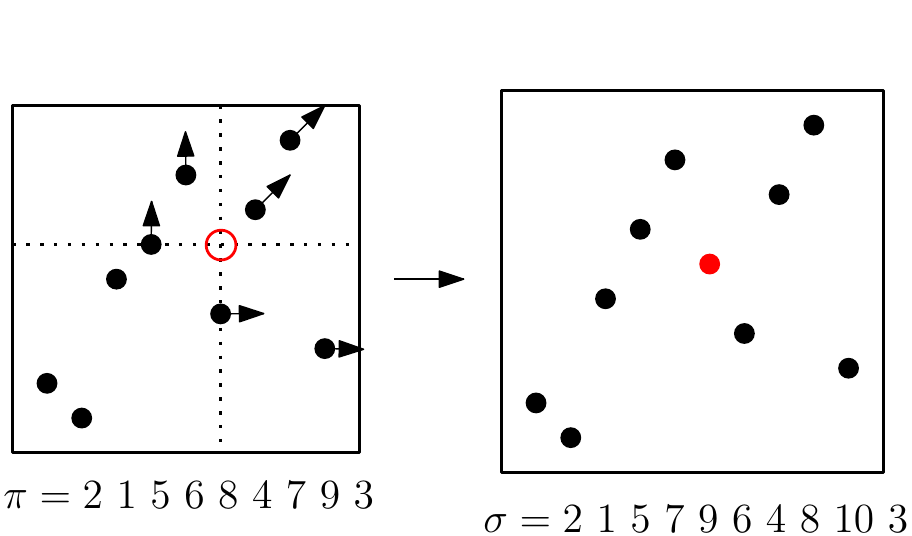}\\
			\caption{Insertion of the point $(6,6)$ (highlighted with a red circle) in the permutation $\pi=215684793$.\label{fig:insert_6_6}}
		\end{center}
	\end{figure}
\end{example}

For a sequence of points, $J=\{(i_\ell,j_\ell)\}_{\ell=1}^k$, and a permutation $\sigma$ of size $n$,   we call $J$ a \emph{valid insertion sequence} for $\sigma$ if $(i_\ell,j_\ell) \in [n+\ell]^2$ for $\ell \in [k].$  A valid insertion sequence, $J$, gives a corresponding sequence of permutations, $(\sigma^0,\cdots, \sigma^k)$, defined by $\sigma^0 = \sigma$ and $\sigma^{\ell}=\insrt(\sigma^{\ell-1},(i_\ell,j_\ell))$ for $1\leq \ell \leq k$.  We denote the final permutation obtained in the sequence by $\insrt(\sigma,J)=\sigma^k.$

Similarly, for a permutation $\rho\in \SG_{n+k}$ and a sequence, $J'=\{(i'_\ell,j'_\ell)\}_{\ell=1}^k$ with $(i'_\ell,j'_\ell)\in [n+k+1-\ell]^2$, we say $J'$ is a \emph{valid deletion sequence} if there is a sequence of permutations $(\rho^0, \cdots, \rho^k)$ with $\rho^0 = \rho$ and, for $1\leq \ell \leq k$, $\rho^{\ell}=\delete(\rho^{\ell-1},(i'_\ell,j'_\ell)).$  If for some $\ell$, $(i'_\ell,j'_\ell)$ is not a valid deletion in $\rho^{\ell-1}$, we say $J'$ is an \emph{invalid deletion sequence} for $\rho$.  If $J'$ is a valid sequence of deletions we let $\delete(\rho,J') = \rho^k$.  If $J$ is an insertion sequence for $\sigma$ and $\pi = \insrt(\sigma,J)$, then the reverse of $J$, denoted $\rev(J)$, is a valid deletion sequence for $\pi$, with $\delete(\pi,\rev(J)) = \sigma.$  

We say an insertion, $(i,j)$, is \emph{internal} for $\sigma$ if the point $(i,j)$ is internal in $\insrt(\sigma,(i,j))$.  A sequence of insertions, $J$, is internal for $\sigma$ if for each $1\leq \ell \leq |J|$ the point $(i_\ell,j_\ell) \in J$ is internal in $\sigma^\ell.$  If $J$ is internal for $\sigma$, the corresponding sequence $(\sigma^0,\cdots, \sigma^{|J|})$ has permutations whose external points are exactly the appropriate shifts of the external points of $\sigma$. In particular, if $\sigma$ is a square permutation, then the external points of the permutations in the sequence are exactly the points of $\sigma$.  Lastly, we say a deletion sequence is internal if every deletion in the sequence comes from an internal point of the corresponding permutation in the sequence. 

For a permutation $\sigma$, let $I(\sigma)$ denote the set of possible internal insertions for $\sigma$ and let $\cJ(\sigma,k)$ denote the set of possible internal insertion sequences of length $k$ for $\sigma$.

\begin{lemma}\label{graves}
	
	Let $\pi \in \SG_{n+k}$ and let $M$ be a collection of $k$ marked points in $\pi$.  There are precisely $k!$ deletion sequences starting from $\pi$ that remove only the $k$ marked points.
	
\end{lemma}

\begin{proof}
	The order in which the points of $M$ are removed uniquely determines the deletion sequence.  There are $k!$ possible orders. 
\end{proof}

\begin{corollary}\label{graves_int}
	Let $\pi \in \asqnk$ and let $\sigma = \ext(\pi)$.  There are precisely $k!$ internal insertion sequences, $J$, such that $\pi = \insrt(\sigma,J).$
\end{corollary}

It is important to highlight that the insertion sequences in Corollary \ref{graves_int} are internal.  For example, the permutation $\sigma = 4132$ has a unique internal insertion (the point $(3,3)$) that gives the permutation $51342$, while the insertion of $(4,4)$ in $4132$ also gives $51342$ but is not internal (see Fig.~\ref{fig:Double_poss_ins}).

\begin{figure}[htbp]
	\begin{center}
		\includegraphics[scale=1]{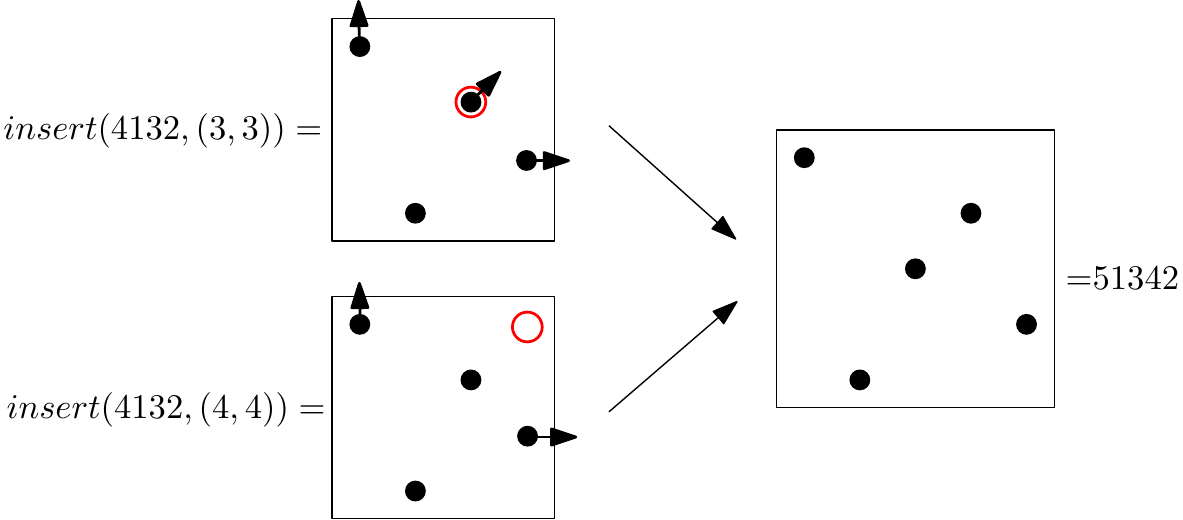}\\
		\caption{Two different insertions (highlighted with a red circle) that give the same permutation. The first insertion is internal, the second one it is not.\label{fig:Double_poss_ins}}
	\end{center}
\end{figure}

\section{Projection for square permutations and Petrov conditions}
\label{sect:perm_to_anchored_seq}
We recall in this section some results from \cite{borga2019square} that are useful for the next sections.

\subsection{Projections for square permutations}
We recall the following key definition.
\begin{definition}
	An \emph{anchored pair of sequences} of size $n$ is a triplet $(X,Y,z_0)$, where $X\in\{U,D\}^n$,
	$Y\in\{L,R\}^n$ and $z_0 \in [n].$ We say that the pair $(X,Y)$ is anchored at $z_0$.
\end{definition}
Given a square permutation $\sigma\in Sq(n),$ we associate to it an anchored pair of sequences $(X,Y,z_0)$ of size $n$ (\emph{cf.}\ Fig.~\ref{Square_perm_sampling_example}) where the labels of $(X,Y)$ are determined by the record types (the sequence $X$ records if a point is a maximum ($U$) or a minimum ($D$) and the sequence $Y$ records if a point is a left-to-right record ($L$) or a right-to-left record ($R$)) and the anchor $z_0$ is equal to the value $\sigma^{-1}(1)$. As a convention, if a point is both a maximum and a minimum (resp.\ a left-to-right and a right-to-left record) we assign a $D$ (resp.\ a $L$).  For a precise and rigorous definition we refer to \cite[Section 2]{borga2019square}.

We denote with $\phi$ the injective\footnote{The injectivity of the map $\phi$ was proved in \cite{duchi_square2}.} map that associates to every square permutation the corresponding anchored pair of sequences, therefore
$$\phi:Sq(n)\to\{U,D\}^n\times\{L,R\}^n\times[n].$$

\begin{figure}[htbp]
	\begin{center}
		\includegraphics[scale=0.5]{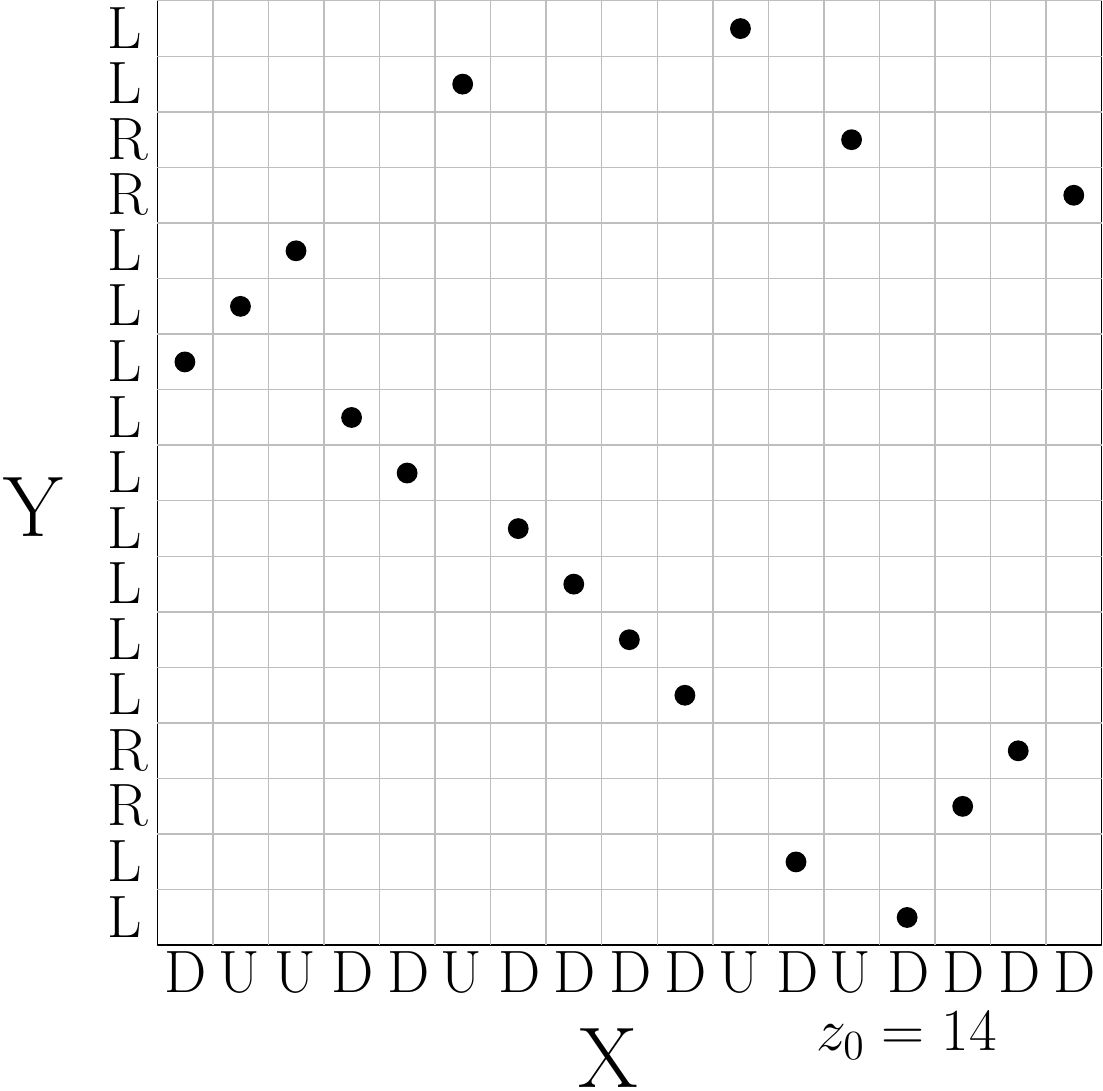}
		\caption{A square permutation $\sigma$ with the associated anchored pair of sequences $\phi(\sigma)=(X,Y,z_0).$  The sequence $X$ is under the diagram (read from left to right) of the permutation and the sequence $Y$ on the left (read from bottom to top).}
		\label{Square_perm_sampling_example}
	\end{center}
\end{figure}

We say that an anchored pair of sequences $(X,Y,z_0)$ of size $n$ is \emph{good} if $X_1 = X_n = X_{z_0}=D$ and $Y_1 = Y_n = L$. Note that $\phi(Sq(n))$ is contained in the set of good anchored pairs of size $n$. 
Note also that the total number of possible good anchored pairs $(X,Y,z_0)$ of size $n$ is
\begin{equation}\label{karl}
2^{n-2}(2\cdot 2^{n-2}+(n-2)\cdot 2^{n-3})= 2(n+2)4^{n-3}.	
\end{equation}

This map $\phi$ is not surjective, but we can identify subsets of good anchored pairs of sequences (called \emph{regular}) and of square permutations where the projection map is a  bijection. In order to do that we need to introduce the \emph{Petrov conditions}.

\subsection{Petrov conditions}

Let $X\in\{U,D\}^n$ and $Y\in\{L,R\}^n$.
Let $\ctd(i)$ denote the number of $D$s in $X$ up to (and including) position $i$.  Similarly define $\ctu(i)$, $\ctl(i)$ and $\ctr(i)$ for the number of $U$s in $X$ and the number of $L$s or $R$s in $Y$, respectively.  Let $\pd(i)$ denote the position of the $i$-th $D$ in $X$ with $\pd(i) = n$ if there are fewer than $i$ indices labeled with $D$ in $X$.  Similarly define $\pu(i)$, $\pl(i)$ and $\pr(i)$ for the location of the indices of the other labels.

\begin{definition}[Petrov conditions]\label{defn:petrov}
	We say that the labels $D$ in $X$ satisfy the Petrov conditions if the following are true:	
	\begin{enumerate}
		\item $|\ctd(i) - \ctd(j) - \frac12(i-j) | < n^{.4}$, for all $|i-j| < n^{.6}$;
		\item $|\ctd(i) - \ctd(j) - \frac12(i-j) |	 < \frac{1}{2}|i-j|^{.6}$, for all $|i-j|> n^{.3}$;
		\item $|\pd(i) - \pd(j) - 2(i-j)| < n^{.4}$, for all $|i-j| < n^{.6}$ and $i,j \leq \ctd(n)$;
		\item $|\pd(i) - \pd(j) - 2(i-j)| < 2|i-j|^{.6}$,	 for all $|i-j|> n^{.3}$ and $i,j \leq \ctd(n)$.
	\end{enumerate}		
\end{definition}

A similar definition holds for the labels $U$ in $X$ and the labels $L$ and $R$ in $Y$ for the functions $\ctu,\ctl,\ctr,$ and $\pu, \pl, \pr$.  We say the Petrov conditions hold for the pair of sequences $X$ and $Y$ if the Petrov conditions hold for the four types of labels of $X$ and $Y$.

\bigskip

Given a permutation $\sigma\in\sq(n)$, we say that $\sigma$ has \emph{regular projections} if the Petrov conditions hold for the corresponding pair of sequences $X$ and $Y$.
For each $z_0 \in [n]$, let $\mathcal{R}(z_0)$ denote the subset of $\sq(n)$ consisting of permutations anchored at $z_0$ and having regular projections. 

Given a sequence $k=k_n=o(n)$ (resp.\ $k=k_n=o(\sqrt{n})$), we fix now some sequence $\delta_n=\delta_n(k)$ such that
\begin{equation}\label{eq:cond_dn}
\delta_n=o(n),\quad \delta_n\geq n^{.9},\quad\text{and}\quad k=o(\delta_n)\quad(\text{resp.\ } k=o(\sqrt{\delta_n})).
\end{equation} 
Note that this is always possible taking for example $\delta_n=\max\{\sqrt{nk},n^{.9}\}$ (resp.\ $\delta_n=\max\{\sqrt{n}k,n^{.9}\}$).

Let $\irr$ denote the permutations of $\sq(n)$ that either do not have regular projections or such that the anchor $z_0$ belongs to $[n] \backslash (\delta_n,n-\delta_n)$. 

For $z_0$ in $(\delta_n,n-\delta_n)$, a uniform square permutation anchored at $z_0$ is in $\irr$ with probability at most $Ce^{-n^c}$ for some positive constants $c$ and $C$ independent of $z_0$ (this follows from the classical bounds for Petrov conditions, see for instance the proof of \cite[Lemma 3.4]{borga2019square}). Thus, for $n$ large enough, we have the following bound
\begin{equation}
\label{eq:irr_bound}
|\irr|\leq 2\delta_n4^{n-2}+2n4^{n-2}Ce^{-n^c}\leq 2\delta_n4^n.
\end{equation}

\bigskip 
 
We also say that a good anchored pair of sequences $(X,Y,z_0)$ is \emph{regular}\footnote{In \cite{borga2019square} a regular anchored pair of sequences $(X,Y,z_0)$ satisfies $z_0\in(n^{.9},n-n^{.9})$ instead of $z_0\in(\delta_n,n-\delta_n)$. One can check that all the statements of \cite{borga2019square} are also true for this slightly more general definition.} if the Petrov conditions hold for $X$ and $Y$ and $z_0\in(\delta_n,n-\delta_n)$.
In \cite[Section 3.2]{borga2019square} we constructed a simple algorithm to produce a square permutation from regular anchored pairs of sequences and we showed that the map $\phi^{-1}$ restricted to this set is bijective. We also showed (see \cite[Lemma 3.8]{borga2019square}) that asymptotically almost all square permutations can be constructed from regular anchored pairs of sequences, thus a permutation sampled uniformly from the set of regular anchored pairs of sequences will produce, asymptotically, a uniform square permutation.

\section{Asymptotic enumeration of almost square permutations}\label{sect:asym_enum}

For $\sigma\in \sq(n)$, let $\asq(\sigma,k)$ denote the set of permutations in $\asqnk$ that are of the form $\insrt(\sigma,J)$ for some valid insertion sequence $J$ that is internal with respect to $\sigma$.  For a collection of permutations $\mathcal{S} \in \sq(n)$, let $\asq(\mathcal{S},k) = \bigcup_{\sigma\in \mathcal{S}} \asq(\sigma,k),$ i.e.\ the set of permutations in $\asqnk$ whose exterior lies in $\mathcal{S}$.     
By Corollary \ref{graves_int}, for each $\pi \in \asq(\sigma,k)	$, there are exactly $k!$ internal insertion sequences $J \in \cJ(\sigma,k)$ such that $\insrt(\sigma,J) = \pi$.  These are the only ways to reach $\pi$ by an internal insertion sequence from a square permutation.  Thus

\begin{equation}\label{bumping}
	|\asqnk| = \sum_{\sigma \in \sq(n)}|\asq(\sigma,k)| = \sum_{\sigma\in \sq(n)}\frac{1}{k!}|\cJ(\sigma,k)|.	
\end{equation}

We proceed by finding upper and lower bounds for $|\cJ(\sigma,k)|$ for $\sigma\in \sq(n)$.    
The following lemma gives bounds on the size of $I(\sigma)$ for $\sigma\in\regperm(z_0)$ for some $z_0 \in (\delta_n,n-\delta_n)$.

\begin{lemma}\label{regbound}
	There exists $c>0$ such that for every $z_0\in (\delta_n,n-\delta_n)$, and every $\sigma \in \regperm(z_0)$,
	\begin{equation}\label{regeq}
		2(z_0 - cn^{.6})(n-z_0-cn^{.6})\leq |I(\sigma)| \leq 			2(z_0 + cn^{.6})(n-z_0 +cn^{.6}).
	\end{equation}
	As a consequence, there exists $\varepsilon>0$ such that $|I(\sigma)|\geq \varepsilon n\delta_n$, for every $z_0\in (\delta_n,n-\delta_n)$ and every $\sigma \in \regperm(z_0)$.
\end{lemma}

\begin{proof}
	
	By \cite[Lemma 3.6]{borga2019square}, square permutations with regular projections anchored at $z_0\in (\delta_n,n-\delta_n)$ have points which are contained between the lines
	\begin{itemize}
		\item $x+y = z_0 \pm cn^{.6}$; 
		\item $x+y = 2n - z_0 \pm cn^{.6}$;
		\item $x-y = z_0 \pm cn^{.6}$;
		\item $y-x = z_0 \pm cn^{.6}$.	
	\end{itemize}
	The upper and lower bounds of \eqref{regeq} are given by the area of the smallest and largest rectangle given by these bounded lines. 
	
	The existence of $\varepsilon>0$ such that $|I(\sigma)|\geq \varepsilon n\delta_n$, for every $z_0\in (\delta_n,n-\delta_n)$ and every $\sigma \in \regperm(z_0)$, is a consequence of \eqref{regeq}.
\end{proof}

Note that, for $\sigma \in \irr$, we have the bound
\begin{equation}\label{irreq}
0\leq |I(\sigma)| \leq (n+1)^2.
\end{equation}

For insertion sequences, the number of possible insertions at each step in the sequence increases. 

\begin{lemma}\label{insert_growth}
	Let $\sigma\in\mathcal{S}_n$. Let $\sigma'$ be obtained by the insertion in $\sigma$ of some point in $I(\sigma)$.  Then 
	$$|I(\sigma)| \leq |I(\sigma')| \leq |I(\sigma)| + 2n.$$  
\end{lemma}

\begin{proof}
	For the lower bound let $(i,j)$ be a point in $I(\sigma)$ that is inserted into $\sigma$ and let $\sigma' = \insrt(\sigma,(i,j)).$  For $(i',j')\in I(\sigma)$ with  $i'<i$ and $j'<j$ then $(i',j')\in I(\sigma')$. On the other hand, if $(i',j')\in I(\sigma)$ with $i'\geq i$ and/or $j' \geq j,$ the points $(i'+1,j')$, $(i',j'+1)$ or $(i'+1,j'+1)$ are in $I(\sigma')$ depending on whether $i'>i, j'>j$ or both.  Thus for every point in $I(\sigma)$ there is a unique corresponding point in $I(\sigma')$.  All other points of $I(\sigma')$ will have the form $(i,j')$ or $(i',j)$, giving the upper bound since there are at most $2n$ such points.    
\end{proof}

Thus subsequent insertions give the following bounds on the size of $\cJ(\sigma,k).$

\begin{lemma}\label{insert_bound}
	Let $\sigma\in\mathcal{S}_n$ and $k\in\Z_{>0}$. It holds that
	\begin{equation} \label{insertion_sequence_bound}
		|I(\sigma)|^k \leq |\cJ(\sigma,k)| \leq ( |I(\sigma)| + 2nk + k^2)^k.
	\end{equation} 
	
\end{lemma}

\begin{proof}
	Let $i\leq k$ and $\sigma^i$ be a permutation obtained by the insertions of $i$ internal point in $\sigma$. From Lemma \ref{insert_growth} we have
	$$|I(\sigma)| \leq |I(\sigma^i)| \leq |I(\sigma)| + 2n+\dots+2(n+i-1)=|I(\sigma)| + i(2n+i-1).$$
	Therefore
	$$|I(\sigma)|^k \leq |\cJ(\sigma,k)|  \leq \prod_{i=1}^{k} ( |I(\sigma)| + i(2n+i-1) ).$$
	Noting that $i(2n+i-1)\leq 2nk + k^2$, for all $i\leq k$, we obtain $\prod_{i=1}^{k} ( |I(\sigma)| + i(2n+i-1) )\leq ( |I(\sigma)| + 2nk + k^2)^k$ and we conclude the proof.
\end{proof}

Fix now $\varepsilon >0$ and assume that $|I(\sigma)| \geq \varepsilon n\delta_n$. Then for $k\leq n,$
\begin{equation}\label{eq:boundbound}
(|I(\sigma)| + 2nk +k^2)^k \leq |I(\sigma)|^{k}\left( 1 + \frac{2nk+k^2}{\varepsilon n\delta_n}\right)^k \leq |I(\sigma)|^k\cdot\exp\left(\tfrac{3k^2}{\varepsilon \delta_n}\right),
\end{equation}
where in the last equality we used that $(1+x)\leq e^x$ and $k\leq n$.
Furthermore for $k = o(\sqrt{\delta_n})$,
\begin{equation}
	\label{eq:bound}
	(|I(\sigma)| + 2nk +k^2)^k \leq |I(\sigma)|^k\cdot\exp(3/\varepsilon\cdot o(1)).
\end{equation}

\begin{corollary}\label{jarjar} Let $\sigma\in\mathcal{S}_n,$ $\varepsilon >0$, $k = o(\sqrt n)$. If $|I(\sigma)| \geq \varepsilon n\delta_n$ then
	 $$|\asq(\sigma,k)| \sim \frac{|I(\sigma)|^k}{k!}\;.$$ 
\end{corollary}

\begin{proof}
	Note that from Corollary \ref{graves_int}, $|\asq(\sigma,k)| = \frac{1}{k!}|\cJ(\sigma,k)|.$ We conclude combining the bounds in \eqref{insertion_sequence_bound} and \eqref{eq:bound} (recalling that from \eqref{eq:cond_dn} we have that $k = o(\sqrt{\delta_n})$ since $k = o(\sqrt n)$).
\end{proof}

We can now prove the main result of this section, that is, when $k=o(\sqrt{n})$ then 
$$|\asqnk| \sim \frac{k!2^{k+1}n^{2k+1}4^{n-3}}{(2k+1)!}.$$

\begin{proof}[Proof of Theorem \ref{approx_size}]
	
	For $\sigma \in \irr$, we have from \eqref{irreq} the rough bounds of $$0\leq |\cJ(\sigma,k)| \leq (n+k)^{2k}.$$  Therefore the contribution from $\irr$ to $\asqnk$ is bounded above by
	
	\begin{equation}\label{eq:irrbound}
		\sum_{\sigma \in \irr}|\asq(\sigma,k)|=\sum_{\sigma\in \irr} \frac{1}{k!}|\cJ(\sigma,k)| \leq \frac{1}{k!}2\delta_n4^n(n+k)^{2k} \leq \frac{1}{k!}2\delta_n4^n2^{2k}n^{2k},
	\end{equation}
	where in the first inequality we also used the bound for $|\irr|$ obtained in \eqref{eq:irr_bound} and in the second inequality the fact that $k=o(\sqrt n)$.

	We now focus on the contribution of 
	\begin{equation*}
	\sum_{\sigma \in \sq(n)\setminus\irr}|\asq(\sigma,k)|=\sum_{z_0 \in (\delta_n,n-\delta_n)} \sum_{\sigma\in\mathcal{R}(z_0)}\frac{1}{k!}|\cJ(\sigma,k)|.
	\end{equation*}
	Using the bounds in Lemma \ref{insert_bound}, we have
	\begin{multline}\label{eq:logbound_1}
	\sum_{z_0 \in (\delta_n,n-\delta_n)} \sum_{\sigma\in\mathcal{R}(z_0)}\frac{1}{k!}|I(\sigma)|^k	\\
	\leq\sum_{\sigma \in \sq(n)\setminus\irr}|\asq(\sigma,k)|\leq\\
	\sum_{z_0 \in (\delta_n,n-\delta_n)} \sum_{\sigma\in\mathcal{R}(z_0)}\frac{1}{k!}\left(|I(\sigma)|+2nk+k^2\right)^k	.
	\end{multline}
	From Lemma \ref{regbound} we know that there exists $\varepsilon>0$ such that $|I(\sigma)|\geq \varepsilon n\delta_n$, for every $z_0\in (\delta_n,n-\delta_n)$ and every $\sigma \in \regperm(z_0)$. Therefore, using \eqref{eq:bound}, the right-hand side of the above inequality is bounded by
	\begin{equation}\label{eq:logbound_2}
	\exp(3/\varepsilon\cdot o(1))\cdot\sum_{z_0 \in (\delta_n,n-\delta_n)} \sum_{\sigma\in\mathcal{R}(z_0)}\frac{1}{k!}|I(\sigma)|^k.	
	\end{equation}
	
	For any fixed anchor $z_0$ in $(\delta_n,n-\delta_n)$, the total number of permutations in $\mathcal{R}(z_0)$ is asymptotically $2\cdot4^{n-3} \cdot (1+o(1))$, where the error term is uniform in $z_0$ (again this follows from the classical bounds for Petrov conditions, see for instance the proof of \cite[Lemma 3.4]{borga2019square}). Using this result and the estimate in Lemma \ref{regbound}, we obtain that
	\begin{equation}\label{reg_contribution}
	\sum_{z_0 \in (\delta_n,n-\delta_n)} \sum_{\sigma\in\mathcal{R}(z_0)}\frac{1}{k!}|I(\sigma)|^k\sim \frac{1}{k!}2^{k+1}n^{2k+1}4^{n-3} \int_0^1 (t(1-t))^kdt,
	\end{equation}
	where we used the standard Riemann integral approximation with the substitution $z_0=\lfloor nt \rfloor$ for $t\in(0,1)$.
	Combining \eqref{eq:logbound_1}, \eqref{eq:logbound_2} and \eqref{reg_contribution}, we conclude that
	\begin{equation}\label{eq:keyapprox}
	\sum_{\sigma \in \sq(n)\setminus\irr}|\asq(\sigma,k)|\sim \frac{1}{k!}2^{k+1}n^{2k+1}4^{n-3} \int_0^1 (t(1-t))^kdt .
	\end{equation}
	
	The integral in \eqref{eq:keyapprox} evaluates to $(k!)^2/(2k+1)!$. Thus combining this with the (negligible) contribution from $\irr$ (since $\delta_n=o(n)$) we have
	\begin{equation*}
	|\asqnk|\sim\frac{k!}{(2k+1)!}2^{k+1}n^{2k+1}4^{n-3}.\qedhere
	\end{equation*}
\end{proof}

A consequence of the proof of \eqref{eq:keyapprox} is the following
\begin{corollary}\label{cor:asympt}
	For $k=o(\sqrt n)$ and $s\in(0,1)$,
	\begin{equation*} 
	\sum_{z_0 \in (\delta_n,ns)}|\asq(\mathcal{R}(z_0),k)|\sim \frac{1}{k!}2^{k+1}n^{2k+1}4^{n-3} \int_0^s (t(1-t))^kdt .
	\end{equation*}
\end{corollary}

We finally investigate the case when $k=o(n),$ proving that
\begin{equation*}
\log\left(|\asq(n,k)|\right)=\log\left(\frac{k!}{(2k+1)!}2^{k+1}n^{2k+1}4^{n-3}\right)+o(k).
\end{equation*}

\begin{proof}[Proof of Theorem \ref{approx_size_2}] Similarly as before, from \eqref{eq:irrbound}, \eqref{eq:logbound_1}, \eqref{eq:boundbound}, \eqref{reg_contribution} and $ \int_0^1 (t(1-t))^kdt=\frac{(k!)^2}{(2k+1)!}$ we have that
	\begin{align*}
	|\asq(n,k)|&=\sum_{\sigma \in \sq(n)\setminus\irr}|\asq(\sigma,k)|+\sum_{\sigma \in \irr}|\asq(\sigma,k)|\\
	&\leq\exp\left(\frac{3k^2}{\varepsilon\delta_n}\right)\cdot\frac{k!}{(2k+1)!}2^{k+1}n^{2k+1}4^{n-3} (1+o(1)).
	\end{align*}
    Applying the logarithm we obtain
	\begin{equation*}
	\log\left(|\asq(n,k)|\right)\leq
	\frac{3k^2}{\varepsilon\delta_n}+\log\left(\frac{k!}{(2k+1)!}2^{k+1}n^{2k+1}4^{n-3}\right) +o\left(1\right),
	\end{equation*}
	and so
	\begin{equation}\label{eq:morebpunds}
	\frac{\log\left(|\asq(n,k)|\right)-\log\left(\frac{k!}{(2k+1)!}2^{k+1}n^{2k+1}4^{n-3}\right)}{k}\leq \frac{\frac{3k^2}{\varepsilon\delta_n}+o(1)}{k}.
	\end{equation}
	On the other hand, using again \eqref{eq:logbound_1}, \eqref{reg_contribution} and  and $ \int_0^1 (t(1-t))^kdt=\frac{(k!)^2}{(2k+1)!}$, we have 
	\begin{equation*}
	|\asq(n,k)|\geq\sum_{\sigma \in \sq(n)\setminus\irr}|\asq(\sigma,k)|\geq
	\frac{k!}{(2k+1)!}2^{k+1}n^{2k+1}4^{n-3} (1+o(1))
	\end{equation*}
	and so
	\begin{equation*}
	\frac{\log\left(|\asq(n,k)|\right)-\log\left(\frac{k!}{(2k+1)!}2^{k+1}n^{2k+1}4^{n-3}\right)}{k}\geq \frac{o(1)}{k}.
	\end{equation*}
	Putting together the last bound with the bound in \eqref{eq:morebpunds} and recalling that when $k=o(n)$ then $\delta_n$ is a sequence such that $k=o(\delta_n)$ we obtain the desired result.
\end{proof}

\section{The permuton limit of almost square permutations}
\label{sect:permutons}
We recall the minimal notions on permutons limits that we need for this section. For a complete introduction to permutons see \cite[Section 2]{bassino2017universal}.

A \emph{permuton} $\mu$ is a Borel probability measure on the unit square $[0,1]^2$ with uniform marginals, that is
\[
\mu( [0,1] \times [a,b] ) = \mu( [a,b] \times [0,1] ) = b-a,
\]
for all $0 \le a \le b\le 1$. Any permutation $\sigma$ of size $n \ge 1$ may be interpreted as a permuton $\mu_\sigma$ given by the sum of Lebesgue area measures
\begin{equation}
\label{eq:perdef}
\mu_\sigma= n \sum_{i=1}^n \Leb\big([(i-1)/n, i/n]\times[(\sigma(i)-1)/n,\sigma(i)/n]\big).
\end{equation}

Let $\mathcal M$ be the set of permutons. We recall that a sequence of (deterministic) permutons $(\mu_n)_n$ converges \emph{weakly} to $\mu$ (simply denoted $\mu_n \to \mu$) if 
$$
\int_{[0,1]^2} f d\mu_n \to \int_{[0,1]^2} f d\mu,
$$
for every bounded and continuous function $f: [0,1]^2 \to \mathbb{R}$. With this topology, $\mathcal M$ is compact and metrizable by the metric $d_{\square}$ defined, for every pair of permutons $(\mu,\mu'),$ by
$$d_{\square}(\mu,\mu')=\sup_{R\in\mathcal{R}}|\mu(R)-\mu'(R)|,$$
where $\mathcal R$ denotes the set of rectangles contained in $[0,1]^2.$ 

The convergence for random permutations is defined as follows.

\begin{definition}
	We say that a random permutation $\bm{\sigma}_n$ converges in distribution to a random permuton $\bm{\mu}$ as $n \to \infty$ if the random permuton $\mu_{\bm{\sigma}_n}$ converges in distribution to $\bm{\mu}$ with respect to the topology defined above.  
\end{definition}

\subsection{Permuton convergence for square permutations with a fixed number of internal points}
\label{sect:muz}
We prove in this section Theorem \ref{fixedk_thm}. We recall here the rigorous construction of the permuton $\mu^z$ mentioned in the introduction.

Let $z$ be a point in $[0,1]$.  Let $L_1$ and $L_4$ denote the line segments with slope $-1$ connecting $(0,z)$ to $(z,0)$ and $(1-z,1)$ to $(1,1-z)$, respectively.  Similarly let $L_2$ and $L_3$ denote the line segments with slope $1$ connecting $(0,z)$ to $(1-z,1)$ and $(z,0)$ to $(1,1-z)$, respectively.  The union of $L_1$, $L_2$, $L_3$ and $L_4$ forms a rectangle in $[0,1]^2.$  
For each of the line segments $L_i$ ($i=1,2,3$, or $4$) we will define a measure $\mu^z_i$ as a rescaled Lebesgue measure. Let $\nu$ be the Lebesgue measure on $[0,1]$.  Let $S$ be a Borel measurable set on $[0,1]^2$.  For each $i$, let $S_i = S\cap L_i$.  Finally let $\pi_x(S_i)$ be the projection of $S_i$ onto the $x$-axis and $\pi_y(S_i)$ the projection onto the $y$-axis.  As each line has slope $1$ or $-1$, the measures of the projections satisfy $\nu(\pi_x(S_i)) = \nu(\pi_y(S_i)).$  For each $i=1,2,3,4$, define $\mu^z_i(S) := \frac{1}{2} \nu( \pi_x( S_i ) ) = \frac12 \nu( \pi_y(S_i)).$
Finally we define the measure $\mu^z = \mu^z_1+\mu^z_2+\mu^z_3+\mu^z_4.$ The measure $\mu^z$ is a permuton (see \cite[Lemma 4.2]{borga2019square}).

\bigskip

Before proving our main result we need two technical lemmas.

\begin{lemma}\label{lem:nochanges2}
	Let $\sigma$ be a permutation of size $n$ and $\sigma'$ be a permutation obtained from $\sigma$ by adding a point (not necessarily internal) to the diagram of $\sigma$. Then
	$$d_{\square}(\mu_{\sigma},\mu_{\sigma'})\leq\frac 6 n.$$
\end{lemma}

\begin{proof}
	Fix  a rectangle $R\subset [0,1]^2$. Recall that by definition $\mu_{\sigma}$ is the permuton induced by the sum of area measures on points of $\sigma$ scaled to fit within $[0,1]^2$ (see \eqref{eq:perdef}). Suppose that there are $\ell$ points of $\sigma$ contained in $R$. Therefore, keeping track of the possible area measures intersecting the boundaries of $R$, we have that $|\mu_{\sigma}(R)-\frac \ell n|\leq \frac{2}{n}$. Now, noting that the addition of one point to the diagram of $\sigma$, can change the number of points inside $R$ by at most 2, we obtain that $|\mu_{\sigma'}(R)-\frac \ell n|\leq \frac{4}{n}.$ Therefore $|\mu_{\sigma}(R)-\mu_{\sigma'}(R)|\leq \frac{6}{n}$. Since the latter bound does not depend on the choice of $R$ we can conclude the proof.
\end{proof}

In \cite[Lemma 4.3]{borga2019square} we showed that for $\sigma_n\in Sq(n)\setminus\irr$ the permutons, $\mu_{\sigma_n}$ and $\mu^{z_n}$ with $z_n = \sigma_n^{-1}(1)/n,$ have distance $d_{\square}(\mu_{\sigma_n},\mu^{z_n})$ that tends to zero as $n$ tends to infinity, uniformly over all choices of $\sigma_n$. We prove here that the same result holds for permutations in $ASq(Sq(n)\setminus\irr,k)$ whenever $k=o(n)$.

\begin{lemma} \label{lem:nochanges}
	Let $k=o(n)$. The following limit holds
	\begin{equation*}
		\sup_{\sigma_n\in ASq(Sq(n)\setminus\irr,k)}d_{\square}(\mu_{\sigma_n},\mu^{z_n})\to 0.
	\end{equation*}
	
\end{lemma}

\begin{proof}
	We have the following bound for every $\sigma_n\in ASq(Sq(n)\setminus\irr,k)$
	$$d_{\square}(\mu_{\sigma_n},\mu^{z_n})\leq d_{\square}(\mu_{\sigma_n},\mu_{\ext(\sigma_n)})+d_{\square}(\mu_{\ext(\sigma_n)},\mu^{z_n})$$
	that translates into
	$$\sup_{\sigma_n\in ASq(Sq(n)\setminus\irr,k)}d_{\square}(\mu_{\sigma_n},\mu^{z_n})\leq \sup_{\sigma_n\in ASq(Sq(n)\setminus\irr,k)}d_{\square}(\mu_{\sigma_n},\mu_{\ext(\sigma_n)})+\sup_{\sigma_n\in Sq(n)\setminus\irr}d_{\square}(\mu_{\sigma_n},\mu^{z_n}).$$
	The second term in the right-hand side of the above equation tends to zero thanks to the aforementioned \cite[Lemma 4.3]{borga2019square} and the first term tends to zero because the addition of $k=o(n)$ internal points cannot modify the permuton limit of a sequence of permutations, as shown in Lemma \ref{lem:nochanges2}.
\end{proof}

We can now prove the main result of this section, that is, if $k>0$ is fixed and $\bm\sigma_n$ is uniform in $\asqnk$, then $\mu_{\bm{\sigma}_n} \stackrel{d}{\longrightarrow} \mu^{\zz^{(k)}}$ as $n\to \infty$.

\begin{proof}[Proof of Theorem \ref{fixedk_thm}]

With the asymptotic formula for the cardinality of $\asqnk$ (obtained in Theorem \ref{approx_size}) and Corollary \ref{cor:asympt}  we can determine the distribution for the value of $\bm \sigma^{-1}_n(1)$, for a uniform permutation $\bm{\sigma}_n$ in $\asqnk$ when $k$ is fixed. Specifically,
\begin{equation}\label{coral}
\prob(\bm{\sigma}_n^{-1}(1) \leq ns ) \sim \frac{ \frac{1}{k!}2^{k+1}n^{2k+1}4^{n-3}\int_0^s (t(1-t))^kdt}{\frac{k!}{(2k+1)!}2^{k+1}n^{2k+1}4^{n-3}} =  (2k+1){2k \choose k}\int_{0}^s(t(1-t))^kdt,
\end{equation}
where we used again the fact that the contribution to the numerator of permutations in $\irr$ is negligible w.r.t. the cardinality of $\asqnk$ (see \eqref{eq:irrbound}). Therefore $\zz_n=\frac{\bm{\sigma}_n^{-1}(1)}{n}\stackrel{d}{\longrightarrow}\zz^{(k)}.$ 

The map $z\to\mu^z$ is continuous as a function from $(0,1)$ to $\mathcal{M}$, and thus $\mu^{\zz_n}$ converges in distribution to $\mu^{\zz^{(k)}}.$ By Lemma \ref{lem:nochanges}, and again the fact that $|\irr|$ is negligible w.r.t. $|\asqnk|$ , we also have that $d_{\square}(\mu_{\bm{\sigma}_n},\mu^{\zz_n})$ converges almost surely to zero. Therefore, combining these results,  we can conclude that $\mu_{\bm{\sigma}_n}$ converges in distribution to $\mu^{\zz^{(k)}}.$
\end{proof}

\subsection{Permuton convergence for square permutations with a growing number of internal points}
We prove in this section Theorem \ref{permuton_limit}.
When the number $k$ of internal points tends to infinity, we have the following result.

\begin{lemma}\label{lemm:conc_result}
	Let $k=o(n)$ and assume that $k\to\infty$. Then for $\bm{\sigma}_n$ uniform in $\asqnk$ it holds 
	$$\frac{\bm{\sigma}_n^{-1}(1)}{n}\stackrel{d}{\longrightarrow}1/2.$$
\end{lemma}

\begin{proof} Note that for every $0<\lambda<\frac 1 2,$
	\begin{equation}\label{eq:startingpoint}
		\prob(\bm{\sigma}_n^{-1}(1) \leq n(1/2 - \lambda))\leq\frac{\sum_{z_0 \in (\delta_n,n(1/2 - \lambda))} \sum_{\sigma\in\mathcal{R}(z_0)}\frac{1}{k!}|\cJ(\sigma,k)|+\sum_{\sigma \in \irr}|\asq(\sigma,k)|}{|\asqnk|}.
	\end{equation}
	We focus on the term $\sum_{z_0 \in (\delta_n,n(1/2 - \lambda))} \sum_{\sigma\in\mathcal{R}(z_0)}\frac{1}{k!}|\cJ(\sigma,k)|$. From Lemma \ref{regbound} we know that $|I(\sigma)|\geq \varepsilon n\delta_n$. Therefore, using Lemmas \ref{regbound} and \ref{insert_bound}, the estimates in \eqref{eq:boundbound} and the fact that $|{R}(z_0)|\leq 2^{2n-5}$, we obtain
	\begin{equation*}
	\sum_{z_0 \in (\delta_n,n(1/2 - \lambda))} \sum_{\sigma\in\mathcal{R}(z_0)}\frac{1}{k!}|\cJ(\sigma,k)|\leq\frac{1}{k!}2^{2n-5}\exp\left(\tfrac{3k^2}{\varepsilon \delta_n}\right)\sum_{z_0 \in (\delta_n,n(1/2 - \lambda))}\left(	2(z_0 + cn^{.6})(n-z_0 +cn^{.6})\right)^k.
	\end{equation*}
	We also have the following asymptotic estimate 
	\begin{multline*}
	\frac{1}{k!}2^{2n-5}\exp\left(\tfrac{3k^2}{\varepsilon \delta_n}\right)\sum_{z_0 \in (\delta_n,n(1/2 - \lambda))}\left(	2(z_0 + cn^{.6})(n-z_0 +cn^{.6})\right)^k\\
	\sim\frac{1}{k!}2^{2n-5}\exp\left(\tfrac{3k^2}{\varepsilon \delta_n}\right)2^kn^{2k+1}\int_{0}^{1/2-\lambda}(t(1-t))^kdt,
	\end{multline*}
	where we used again the standard Riemann integral approximation with the substitution $z_0=\lfloor nt \rfloor$ for $t\in(0,1/2-\lambda)$.
	Noting that for $t \leq 1/2 - \lambda$ we have $(t(1-t))^k \leq 4^{-k}(1-4\lambda^2)^k \leq 4^{-k}e^{-4\lambda^2k}$, from the two equations above we obtain that
	\begin{equation}\label{eq:furtherbound}
	\sum_{z_0 \in (\delta_n,n(1/2 - \lambda))} \sum_{\sigma\in\mathcal{R}(z_0)}\frac{1}{k!}|\cJ(\sigma,k)|\leq \frac{1}{k!}2^{2n-5}\exp\left(\tfrac{3k^2}{\varepsilon \delta_n}\right)2^kn^{2k+1}4^{-k}e^{-4\lambda^2k}(1+o(1)).
	\end{equation}
	Using that $|\asqnk| \sim \frac{k!2^{k+1}n^{2k+1}4^{n-3}}{(2k+1)!}$ and the bound  in \eqref{eq:irrbound} we can conclude from \eqref{eq:startingpoint} and \eqref{eq:furtherbound} that
	\begin{equation*}
		\prob(\bm{\sigma}_n^{-1}(1) \leq n(1/2 - \lambda))\leq\frac{(2k+1)e^{-4\lambda^2k}}{\sqrt{k \pi}}\exp\left(\tfrac{3k^2}{\varepsilon \delta_n}\right)(1+o(1)).
	\end{equation*}
	Therefore we can conclude that $\prob(\bm{\sigma}_n^{-1}(1) \leq n(1/2 - \lambda))\to0$, for every $0<\lambda<\frac 1 2$.
	
	Since $\bm{\sigma}_n^{-1}(1)\stackrel{d}{=}n+1-\bm{\sigma}_n^{-1}(1)$ then the probability that $\prob(\bm{\sigma}_n^{-1}(1) \geq n(1/2 + \lambda))$ is also equally small and this concludes the proof.
\end{proof}

\begin{proof}[Proof of Theorem \ref{permuton_limit}]
	The proof is identical to the proof of Theorem \ref{fixedk_thm} above, using the concentration result for $\bm{\sigma}_n^{-1}(1)$ obtained in Lemma \ref{lemm:conc_result}.
\end{proof}

 \section{Insertions in 321-avoiding permutations}
 \label{sect:321av}

Permutations in $\avn(321)$ are in bijection with Dyck paths of size $2n$.  A Dyck path of size $2n$ is a path with two types of steps: $(1,1)$ or $(1,-1)$, that is conditioned to start at $(0,0),$ end at $(2n,0)$, and remain non-negative in between.  There are many possible bijections to choose from between these two sets.  One particular bijection comes from \cite{BJS}, which we refer to as the Billey--Jockusch--Stanley (or BJS) bijection.  For a Dyck path, $\gamma_n,$ of size $2n$, we let $\tau_n = \tau_{\gamma_n}$ denote the corresponding permutation in $\avn(321)$ under the BJS-bijection.  In the other direction, for a permutation $\tau_n \in \avn(321)$ we let $\gamma_n = \gamma_{\tau_n}$ denote the corresponding Dyck path under the inverse bijection.  
This bijection is used in \cite{hoffman2017pattern} to show that the points of a permutation that avoid a decreasing sequence of size three converge to the Brownian excursion when properly scaled.

Specifically, extend the definition of the permutation $\tau_n$ so that $\tau_n(0)=0$ and for $t\in [0,1]$, let 
$$F_{\tau_n}(t) \coloneqq \frac{1}{\sqrt{2n}}\big |\tau_n(\lfloor nt \rfloor) - \lfloor nt \rfloor \big|.$$  
 
\begin{theorem}[Theorem 1.2 in \cite{hoffman2017pattern}]\label{chacha}
Let $\bm{\tau}_n$ be a uniformly random permutation in $\avn(321)$.  Then
$$\left(F_{\bm{\tau}_n}(t)\right)_{t\in [0,1]} \stackrel{d}{\longrightarrow} \left(\bm{e}_t\right)_{t\in [0,1]},$$
where $\bm{e}_t$ is the Brownian excursion on $[0,1]$ and the convergence holds in the space of right-continuous functions $D([0,1],\mathbb{R})$.
\end{theorem}

The main step in the proof of Theorem \ref{chacha} is showing that the function $F_{\bm{\tau}_n}(t)$ is often close to the corresponding scaled Dyck path $\gamma_{\bm{\tau}_n}$, which converges in distribution to the Brownian excursion \cite{Ka76}.  The proof uses an alternative version of the Petrov conditions stated in terms of Dyck paths. We denote the Petrov conditions for Dyck paths with $PC'$ and the Petrov conditions for permutations used in this paper with $PC$ (see Definition \ref{defn:petrov}). $PC'$ can be translated to permutations obtaining a slightly modified version of $PC$. We say $\tau_n\in \avn(321)$ satisfies $PC'$ if, using the BJS bijection, the corresponding Dyck path, $\gamma_n$, satisfies $PC'$.

In what follow we say that $\tau_n\in \avn(321)$ satisfies the Petrov conditions if it satisfies both $PC$ and $PC'$ (and we do the same for Dyck paths). The exact version of $PC'$ is not important for our results, though we point out that a uniform random permutation in $\avn(321)$ has exponentially small probability to satisfy only one set of conditions among $PC$ and $PC'$. 
This together with Corollary 5.5 and Proposition 5.6 and 5.7 of \cite{hoffman2017pattern}, implies that there exist positive constants $C$ and $\delta$ such that the probability that the Petrov conditions are not satisfied for a uniform random permutation $\bm \tau_n$ in $\avn(321)$ is bounded above by $Ce^{-n^\delta}$. 

\begin{lemma}[Lemma 2.7 in \cite{hoffman2017pattern}]\label{flux}
	Let $\gamma_n$ be a Dyck path of size $2n$ that satisfies the Petrov conditions, and let $\tau_n$ be the corresponding permutation in $\avn(321)$.  If $(j,\tau_n(j))$ is a left-to-right maximum, then 
	$$|\tau_n(j) - j - \gamma_n(2j)| \leq  10n^{.4},$$
	and if $(j,\tau_n(j))$ is a right-to-left minimum, then 
	$$| \tau_n(j) - j  + \gamma_n(2j)| \leq 10n^{.4}.$$
	Therefore for all $j\leq n$, 
	$$\gamma_n(2j)- 10n^{.4} \leq | \tau_n(j) - j| \leq \gamma_n(2j) + 10n^{.4}.$$
\end{lemma}

Let $M(\gamma_n)  = \max_{1\leq j\leq n}\gamma_n(2j)$ be the maximum of $\gamma_n$  and $D(\tau_n) = \max_{1\leq j \leq n} | \tau_n(j) - j|$ be the maximum absolute displacement.

\begin{corollary}\label{dyck_max}
	Let $\tau_n$ be a permutation in $\avn(321)$ and let $\gamma_n = \gamma_{\tau_n}$ be the corresponding Dyck path of size $2n$.    If $\tau_n$ (and thus $\gamma_n$) satisfies the Petrov conditions, then $D(\tau_n) \leq M(\gamma_n) + 10n^{.4}.$
\end{corollary}
  
Let $E^+(\tau_n)$ denote the set of left-to-right maxima of $\tau_n$ and $E^-(\tau_n)$ the complement of $E^+(\tau_n)$ (thus the points of $E^-(\tau_n)$ are all the right-to-left minima of $\tau_n$ that are not fixed points).  For $1\leq i \leq n$ let 
$$i^+ = \max_{j\leq i} \{ j: (j,\tau_n(j)) \in E^+(\tau_n)\}.$$  Similarly let $$i^- = 
\min_{j\geq i} \{ j: (j,\tau_n(j)) \in E^-(\tau_n)\},$$ 
with the exception that if $i$ is a fixed point, then $i^-=i$.  

\begin{lemma}\label{sandwich}
	Let $\tau_n\in \avn(321)$ satisfy the Petrov conditions.  For $1\leq i \leq n$,
	$$\Big| |\tau_n(i^+) - i^+ | - |\tau_n(i) - i| \Big| < 25n^{.4},$$
	and 
	$$\Big| |\tau_n(i^-) - i^-| - |\tau_n(i) - i| \Big|< 25n^{.4}.$$
\end{lemma}

\begin{proof}
	
By \cite[Lemma 2.5]{hoffman2017pattern}, any interval of length $n^{.3}$ must contain both a point in $E^+(\tau_n)$ and $E^-(\tau_n)$, and therefore 
\begin{equation}\label{eq:swedish}
\max\{i-i^+, i^- -i\} \leq n^{.3}.
\end{equation}
By the Petrov conditions for Dyck paths, if $|x-y| < 2n^{.6}$ then $|\gamma_n(x) - \gamma_n(y)|< n^{.4}.$  Therefore both $|\gamma_n(2i^+) - \gamma_n(2i)| < n^{.4}$ and $|\gamma_n(2i^-) - \gamma_n(2i)| < n^{.4}$.  By Lemma \ref{flux},
 \begin{align*}
	\Big|| \tau_n(i^+) - i^+| - |\tau_n(i) - i|\Big| & \leq |\gamma_n(2i^+) + 10n^{.4} - \gamma_n(2i) + 10n^{.4}|\\
	& \leq |\gamma_n(2i^+) - \gamma_n(2i)| + 20n^{.4}\\
	&< 25n^{.4},
\end{align*}
and similarly, 
 \begin{align*}
	\Big|| \tau_n(i^-) - i^-| - |\tau_n(i) - i|\Big| & \leq |\gamma_n(2i^-) + 10n^{.4} - \gamma_n(2i) + 10n^{.4}|\\
	& \leq |\gamma_n(2i^-) - \gamma_n(2i)| + 20n^{.4}\\
	&< 25n^{.4}.
\end{align*}
This ends the proof.
\end{proof}

\begin{lemma}\label{historique}
Let $\tau_n$ be a permutation in $\avn(321)$ that satisfies the Petrov conditions.  Then 
$$\frac{1}{(2n)^{3/2}}|I(\tau_n)| = \int_0^1 F_{\tau_n}(t) dt + O(n^{-.1}).$$	
\end{lemma}

\begin{proof}

Since $(i,j) \in I(\tau_n)$ if and only if $\tau_n(i^-)< j \leq \tau_n(i^+)$ then 
$$|I(\tau_n)| = \sum_{i=1}^n \tau_n(i^+) - \tau_n(i-).$$
For each $i$ we may use Lemma \ref{sandwich} and \eqref{eq:swedish} to obtain the upper bound
\begin{align*}
\tau_n(i^+) - \tau_n(i^-) &= (\tau_n(i^+) -i^+) + (i^+ - i^-) + (i^- - \tau_n(i^-))\\
& \leq |\tau_n(i) - i| + 25n^{.4} + 2n^{.3}+  |\tau_n(i) - i| + 25n^{.4}\\
& \leq 2|\tau_n(i) - i| + 100n^{.4}	
\end{align*}
as well as the lower bound
\begin{equation*}
\tau_n(i^+) - \tau_n(i^-) \geq 2|\tau_n(i) - i| - 100n^{.4}.
\end{equation*}
  
In terms of $F_{\tau_n}(\cdot)$, the above estimate rewrites as $|\tau_n(i^+) - \tau_n(i^-) - 2\sqrt{2n}F_{\tau_n}(i/n)| \leq  100n^{.4}$ and  so 
$$|I(\tau_n)| =  \sum_{i=1}^n \left(2\sqrt{2n}F_{\tau_n}(i/n) + O(n^{.4})\right).$$
For $t\in [\frac i n,\frac{i+1}{n})$, $F_{\tau_n}(t) = F_{\tau_n}(i/n)$ and therefore the above sum can be expressed exactly as an integral plus an error term that is at most $O(n^{1.4})$ giving
$$|I(\tau_n)| = (2n)^{3/2}\int_0^1 F_{\tau_n}(t) dt + O(n^{1.4}).$$
Dividing by $(2n)^{3/2}$ finishes the proof.
\end{proof}

Most permutations in $\avn(321)$ satisfy the Petrov conditions and therefore Lemma \ref{historique} applies to most permutations.  This helps in determining the asymptotic behavior of $\avnk$.

\begin{lemma} \label{city maps}
Fix $k>0$.  Let $\tau_n \in \avn(321)$ satisfy the Petrov conditions.  Then $$\frac{1}{(2n)^{3k/2}}|\cJ(\tau_n,k)|  = \left( \int_0^1 F_{\tau_n}(t) dt \right)^k + O(n^{-.1}) $$
and thus
$$\frac{k!}{(2n)^{3k/2}}|\asq(\tau_n,k)| = \left(\int_0^1 F_{\tau_n}(t)dt \right)^k + O(n^{-.1}).$$	
\end{lemma}

\begin{proof}
This follows exactly from the previous lemma together with Lemma \ref{insert_bound}. 
\end{proof}

\begin{proof}[Proof of Theorem \ref{thin red line}]

Partition $\avn(321)$ into two sets $A_n$ and $B_n$, where permutations in $A_n$ satisfy the Petrov conditions and permutations in $B_n$ do not.  Let $c_n$ denote the $n$-th Catalan number $\frac{1}{n+1}{2n \choose n} \sim \frac{4^n}{\sqrt{2\pi n^3}}.$  Let $a_n$ and $b_n$ denote the size of $A_n$ and $B_n$ respectively.  For a uniform permutation in  $\avn(321)$, the Petrov conditions fail with probability at most $Ce^{-n^\delta}$ for some $C,\delta > 0$, thus we have that $b_n \leq  Ce^{-n^\delta}c_n$.    

For any $\tau_n \in \avn(321)$ we always have the upper bound $\asq(\tau_n,k) \leq (n+k)^{2k}$. Thus the contribution to $|\avnk|$ from permutations with external points in $B_n$, i.e.\ permutations in $\asq(B_n,k),$ is at most $c_n(n+k)^{2k} Ce^{-n^\delta}\leq c_n (2n)^{2k}Ce^{-n^\delta} = o(c_n)$.  Using Lemma \ref{city maps}, we obtain
\begin{align}
|\avnk| &=  \sum_{\tau_n \in \avn(321)} |\asq(\tau_n,k)|\nonumber \\
&= c_n \cdot \expect\left[ |\asq(\bm{\tau}_n,k)| \right] \nonumber\\
&= \frac{(2n)^{3k/2}c_n}{k!}\expect\left [\left ( \int_0^1 F_{\bm{\tau}_n}(t) dt \right)^k \Bigg | \bm \tau_n \in A_n \right]\prob(\bm\tau_n \in A_n) + o\left(c_n n^{3k/2-.1}\right).\label{carlos}
\end{align}

Using that $\prob(\bm\tau_n\in B_n ) \leq Ce^{-n^{\delta}}$ and $F_{\bm\tau_n}(t) \leq n^{1/2}$, we have 
\begin{equation}
\label{eq:bound1}
\expect\left [\left ( \int_0^1 F_{\bm{\tau}_n}(t) dt \right)^k \Bigg | \bm \tau_n \in B_n \right]\prob(\bm\tau_n\in B_n )\leq  Cn^{k/2}e^{-n^\delta}.
\end{equation}
Rewriting the expectation $\expect\left [\left ( \int_0^1 F_{\bm{\tau}_n}(t) dt \right)^k \right]$ as
\begin{equation}
\label{eq:bound2}
\expect\left [\left ( \int_0^1 F_{\bm{\tau}_n}(t) dt \right)^k \Bigg | \bm \tau_n \in A_n \right] \prob(\bm\tau_n\in A_n )+\expect\left [\left ( \int_0^1 F_{\bm{\tau}_n}(t) dt \right)^k \Bigg | \bm \tau_n \in B_n \right] \prob(\bm\tau_n\in B_n )
\end{equation}
we have that convergence of the $k$-th moment of $(\int_0^1 F_{\bm\tau_n}(t)dt \big | \bm\tau_n\in A_n)$ is equivalent to the convergence of the $k$-th moment of $\int_0^1 F_{\bm\tau_n}(t)dt$. Moreover, if the limits exist, they must agree. Suppose this is the case, then \eqref{carlos} becomes

\begin{equation}\label{dealio}
|\avnk| = \frac{(2n)^{3k/2}c_n}{k!}\expect\left [ \left ( \int_0^1 F_{\bm\tau_n}(t) dt \right)^k\right] + o\left(c_n n^{3k/2-.1}\right).	
\end{equation}

It remains to show the existence of the limit of the $k$-th moment of the area $\int_0^1 F_{\bm\tau_n}(t) dt$. We have the simple upper bound
\begin{equation}\label{upside down}
\int_0^1 F_{\bm{\tau}_n}(t) dt \leq \sup_{t\in [0,1]} F_{\bm\tau_n}(t) = \frac{1}{\sqrt{2n}} D( \bm \tau_n ).
\end{equation}
For each $k>0$ and for $n$ large enough, from Corollary \ref{dyck_max}
\begin{align}
\expect\left [ \left( \int_0^1 F_{\bm\tau_n}(t) dt \right)^k  \Bigg | \bm\tau_n \in A_n \right] &\leq \expect\left[ \left(\frac{1}{\sqrt{2n}}D( \bm\tau_n) \right)^k \Bigg | \bm\tau_n \in A_n \right] \nonumber \\
& \leq  \expect \left [ \left( \frac{1}{\sqrt{2n}}(M(\bm\gamma_n) + 10n^{.4})\right)^k \Bigg | \bm\tau_n \in A_n \right ] \nonumber \\
& \leq \expect\left[ \left( \frac{1}{\sqrt{2n}} M(\bm\gamma_n) \right)^k \right ] \frac{(1 + O(n^{-.1}))}{\prob(\bm\tau_n \in A_n)}\nonumber \\
& \leq \frac{2}{\prob(\bm\tau_n \in A_n)}\expect\left[ \left( \frac{1}{\sqrt{2n}} M(\bm\gamma_n) \right)^k \right ].\nonumber
\end{align}
Therefore, from \eqref{eq:bound1} and \eqref{eq:bound2} we obtain the following bound
\begin{equation*}
\expect\left [\left ( \int_0^1 F_{\bm{\tau}_n}(t) dt \right)^k \right]\leq 2\cdot \expect\left[ \left( \frac{1}{\sqrt{2n}} M(\bm\gamma_n) \right)^k \right ]+Cn^{k/2}e^{-n^\delta}.
\end{equation*}

By \cite[Theorem 1]{Khorunzhiy_Marckert} the exponential moment of $(2n)^{-1/2} M(\bm\gamma_n)$ is uniformly bounded in $n$,  thus for any $k>0$, the $k$-th moment of $\int_0^1 F_{\bm\tau_n}(t)dt$ is uniformly bounded in $n$.  This along with the convergence in distribution of $\int_0^1 F_{\bm\tau_n}(t) dt$ to $\int_0^1 \bm e_tdt$ implies convergence of the $k$-th moments:

\begin{equation}\label{final countdown}
\expect\left [\left(\int_0^1 F_{\bm{\tau}_n}(t) dt \right)^k\right] \longrightarrow \expect\left[ \left( \int_0^1 \bm{e}_t dt\right) ^k\right ]
\end{equation}
(see \cite[Theorem~4.5.2]{ChungP}, for instance). 

Dividing both sides of \eqref{dealio} by $(2n)^{3k/2}c_n/k!$ gives
\begin{equation}
\frac{|\avnk|}{(2n)^{3k/2}c_n/k!} 	 = \expect\left [\left(\int_0^1 F_{\bm{\tau}_n}(t) dt \right)^k\right]  + o(1),
\end{equation}
and letting $n$ tend to infinity finishes the proof.
\end{proof}

We conclude this section proving Theorem \ref{thm:fluctuations}. We recall that for a permutation $\tau_n\in Av_n(321)$ (with the convention that $\tau_n(0)=0$) we defined
\begin{equation}\label{eq:first_def}
F_{\tau_n}(t) \coloneqq \frac{1}{\sqrt{2n}}\big |\tau_n(\lfloor nt \rfloor) - \lfloor nt \rfloor \big|, \quad t\in[0,1].
\end{equation}
We also generalized this definition, by setting, for a permutation $\tau^k_n\in \avnk$ (with the convention that $\tau^k_n(0)=0$), 
\begin{equation}\label{eq:second_def}
F_{\tau^k_n}(t) \coloneqq \frac{1}{\sqrt{2(n+k)}}\big |\tau^k_n(s(t)) - s(t) \big|,\quad  t\in[0,1],
\end{equation}
where $s(t)=\max\left\{m\leq \lfloor (n+k)t \rfloor|\tau^k_n(m)\text{ is an external point}\right\}$. Note that, for permutations in $ Av_n(321)$, the definition given in \eqref{eq:second_def} coincides with the definition given in \eqref{eq:first_def}.

We need the following technical result.

\begin{lemma}\label{lem:fluctuations}Let $Reg_n^k$ be the set of permutations in $\avnk$ such that the exterior satisfies the Petrov conditions. As $n\to\infty,$
	$$\sup_{\tau^k_n\in Reg^k_n}||F_{\tau^k_n}(t)-F_{\ext(\tau^k_n)}(t)||_{\infty}\to 0,$$
where, for a function $f:[0,1]\to\R$, we denote $||f||_{\infty}=\sup_{t\in[0,1]}|f(t)|$.
\end{lemma}

\begin{proof} 
	Fix $t\in[0,1]$ and  $\tau^k_n\in Reg_n^k$. Set $\tau_n=\ext(\tau^k_n)$. When we add an internal point to a permutation, we shift the points of the permutation diagram above and/or to the right by at most one cell. So, there exist two integers $m(t)$ and $\ell(t)$ such that
	$$\tau_n^k(s(t))=\tau_n(m(t))+\ell(t), \quad\text{with}\quad |m(t)-s(t)|\leq k\text{ and }|\ell(t)|\leq k.$$ 
	Therefore
	\begin{align*}
	\left|F_{\tau^k_n}(t)-F_{\tau_n}(t)\right|&=\left|\frac{1}{\sqrt{2(n+k)}}\big |\tau^k_n(s(t)) - s(t) \big| -\frac{1}{\sqrt{2n}}\big |\tau_n(\lfloor nt \rfloor) - \lfloor nt \rfloor \big|  \right|\\
	&=\left|\frac{1}{\sqrt{2(n+k)}}\big |\tau_n(m(t))+\ell(t) - s(t) \big| -\frac{1}{\sqrt{2n}}\big |\tau_n(\lfloor nt \rfloor) - \lfloor nt \rfloor \big|  \right|\\
	&\leq \frac{1}{\sqrt{2n}}\bigg|\big |\tau_n(m(t))- m(t) \big|  -\big |\tau_n(\lfloor nt \rfloor) - \lfloor nt \rfloor \big|  \bigg|+\frac{2k}{\sqrt{2n}}.
	\end{align*}
	Let $\gamma_n$ be the Dyck path corresponding to $\tau_n$. By Lemma \ref{flux},
	\begin{equation*}
		\Big|\big |\tau_n(m(t))-m(t) \big|  -\big |\tau_n(\lfloor nt \rfloor) - \lfloor nt \rfloor \big|  \Big|\leq|\gamma_n(m(t))-\gamma_n(\lfloor nt \rfloor)|+20n^{.4}.
	\end{equation*}
	By the Petrov conditions for Dyck paths, if $|x-y| < 2n^{.6}$ then $|\gamma_n(x) - \gamma_n(y)|< n^{.4}.$ Noting that $|m(t)-\lfloor nt \rfloor|\leq|m(t)-s(t)|+|s(t)-\lfloor nt \rfloor|\leq 3k$, then we obtain that, for $n$ large enough,
	\begin{equation*}
	\Big|\big |\tau_n(m(t))-m(t) \big| -\big |\tau_n(\lfloor nt \rfloor) - \lfloor nt \rfloor \big|  \Big|\leq 25n^{.4}
	\end{equation*}
	and so
	\begin{align*}
	\left|F_{\tau^k_n}(t)-F_{\tau_n}(t)\right|\leq\frac{25n^{.4}}{\sqrt{2n}}\to 0.
	\end{align*}
	This bound is independent of $t$ and $\tau_n^k$, concluding the proof. 
\end{proof}

\begin{proof}[Proof of Theorem \ref{thm:fluctuations}]
	It is enough to show that for every continuous bounded functional \break
	$G:D([0,1],\mathbb{R})\to\R,$
	\begin{equation*}
	\E\left[G\left(F_{\bm{\tau}^k_n}(t)\right)\right]\to\E\left[G\left(\bm{e}^k_t\right)\right].
	\end{equation*}
	Note that
	\begin{multline*}
	\left|\E\left[G\left(F_{\bm{\tau}^k_n}(t)\right)\right]-\E\left[G\left(\bm{e}^k_t\right)\right]\right|\\
	\leq\E\left[\left|G\left(F_{\bm{\tau}^k_n}(t)\right)-G\left(F_{\ext(\bm{\tau}^k_n)}(t)\right)\right|\right]+\left|\E\left[G\left(F_{\ext(\bm{\tau}^k_n)}(t)\right)\right]-\E\left[G\left(\bm{e}^k_t\right)\right]\right|.
	\end{multline*}
We first show that
	\begin{equation}\label{eq:gooal1}
	\E\left[\left|G\left(F_{\bm{\tau}^k_n}(t)\right)-G\left(F_{\ext(\bm{\tau}^k_n)}(t)\right)\right|\right]\to 0.
	\end{equation}
	We have that
	\begin{equation*}
	\E\left[\left|G\left(F_{\bm{\tau}^k_n}(t)\right)-G\left(F_{\ext(\bm{\tau}^k_n)}(t)\right)\right|\right]=\sum_{\tau^k_n\in\avnk}\left|G\left(F_{\tau^k_n}(t)\right)-G\left(F_{\ext(\tau^k_n)}(t)\right)\right|\P\left(\bm{\tau}^k_n=\tau^k_n\right).
	\end{equation*}
	 The continuity of $G$ and Lemma \ref{lem:fluctuations} show that the contribution to the sum vanishes as $n\to\infty$ for $\tau_n^k\in Reg_n^k$.  Since $G$ is bounded and $\P(\bm{\tau}^k_n\notin Reg_n^k)\to 0$, we can conclude that the contribution to the sum for $\tau_n^k\notin Reg_n^k$ also vanishes as $n\to\infty$, and thus (\ref{eq:gooal1}) holds.
	
	It remains to prove that 
	\begin{equation}\label{eq:goaal2}
		\left|\E\left[G\left(F_{\ext(\bm{\tau}^k_n)}(t)\right)\right]-\E\left[G\left(\bm{e}^k_t\right)\right]\right|\to 0.
	\end{equation}
	Note that 
	\begin{equation*}
	\E\left[G\left(F_{\ext(\bm{\tau}^k_n)}(t)\right)\right]=\sum_{\tau^k_n\in\avnk}G\left(F_{\ext(\tau^k_n)}(t)\right)\cdot\P\left(\bm{\tau}^k_n=\tau^k_n\right).
	\end{equation*}
	From Theorem \ref{thin red line} we have that, uniformly for every $\tau^k_n\in\avnk$,
	\begin{equation*}
	\P\left(\bm{\tau}^k_n=\tau^k_n\right)\sim \frac{k!}{(2n)^{3k/2}}\cdot\frac{1}{c_n}\cdot\mathbb{E}\left[\left(\int_0^1\bm e_t dt\right)^k\right]^{-1},
	\end{equation*}
	and so, setting $Ar_k=\mathbb{E}\left[\left(\int_0^1\bm e_t dt\right)^k\right]^{-1}$, we obtain
	\begin{align*}
	\E\left[G\left(F_{\ext(\bm{\tau}^k_n)}(t)\right)\right]\sim& Ar_k\cdot\sum_{\tau^k_n\in\avnk}G\left(F_{\ext(\tau^k_n)}(t)\right)\cdot\frac{k!}{(2n)^{3k/2}}\cdot\frac{1}{c_n}\\
	&=Ar_k\cdot\sum_{\sigma_n\in Av_n(321)}G\left(F_{\sigma_n}(t)\right)\cdot\left|\asq(\sigma_n,k)\right|\cdot\frac{k!}{(2n)^{3k/2}}\cdot\frac{1}{c_n}.
	\end{align*}
	From Lemma \ref{city maps}, for every $\sigma_n\in Av_n(321)$ that satisfies the Petrov conditions, it holds that
	$$\left|\asq(\sigma_n,k)\right| = \frac{(2n)^{3k/2}}{k!}\left(\left(\int_0^1 F_{\sigma_n}(t)dt \right)^k + O(n^{-.1})\right).$$
	Therefore, using the asymptotic result above and recalling that the number of 321-avoiding permutations that do not satisfy the Petrov conditions is bounded by $Ce^{-n^\delta}c_n$, we obtain
	\begin{align*}
	\E\left[G\left(F_{\ext(\bm{\tau}^k_n)}(t)\right)\right]\sim& Ar_k\cdot\sum_{\sigma_n\in Av_n(321)}G\left(F_{\sigma_n}(t)\right)\cdot \left(\int_0^1 F_{\sigma_n}(t)dt \right)^k\cdot\frac{1}{c_n}\\
	&=Ar_k\cdot\E\left[G\left(F_{\bm\sigma_n}(t)\right)\cdot \left(\int_0^1 F_{\bm\sigma_n}(t)dt \right)^k\right],
	\end{align*}
	where $\bm\sigma_n$ is a uniform permutation in $Av_n(321)$. Using similar arguments to the ones used for proving the result in (\ref{final countdown}), we have that
	\begin{equation*}
	\E\left[G\left(F_{\bm\sigma_n}(t)\right)\cdot \left(\int_0^1 F_{\bm\sigma_n}(t)dt \right)^k\right]\to\E\left[G\left(\bm{e}_t\right)\cdot \left(\int_0^1 \bm{e}_tdt \right)^k\right].
	\end{equation*}
	Finally, recalling the definition of $k$-biased excursion given in Definition \ref{def:kbiasedex}, we can conclude that (\ref{eq:goaal2}) holds, finishing the proof.
\end{proof}

\section*{Acknowledgements}
The authors are very grateful to Mathilde Bouvel and Valentin F\'eray for the various discussions during the preparation of the paper.

The first author is supported by the SNF grant number 200021-172536, ``Several aspects of the study of non-uniform random permutations". The second author is supported by the ANR ``COMBINé'' number 193951. The third author is supported by ERC Starting Grant 680275 ``MALIG".

\bibliographystyle{plain}
\bibliography{pattern}

\end{document}